\documentclass[a4paper,11pt]{emsprocart}
\usepackage[all]{xy}
\usepackage{tikz}
\usetikzlibrary{decorations.pathreplacing}
\usetikzlibrary{matrix,arrows}

\contact[keller@math.jussieu.fr]{Bernhard Keller, Universit\'e Denis Diderot -- Paris~7,
Institut universitaire de France, U.F.R.~de math\'ematiques, Institut de math\'ematiques de Jussieu, U.M.R.~7586
du CNRS, Case 7012, B\^atiment Chevaleret, 75205 Paris, France}

\newcommand{\ie}{{\em i.e.}\ }
\newcommand{\cf}{{\em cf.}\ }
\newcommand{\eg}{{\em e.g.}\ }
\newcommand{\ko}{\: , \;}

\numberwithin{equation}{section}
\newtheorem{theorem}[subsection]{Theorem}
\newtheorem{definition}[subsection]{Definition}
\newtheorem{lemma}[subsection]{Lemma}
\newtheorem{proposition}[subsection]{Proposition}
\newtheorem{corollary}[subsection]{Corollary}
\newtheorem{conjecture}[subsection]{Conjecture}

\newcommand{\reminder}[1]{}

\newcommand{\opname}[1]{\operatorname{\mathsf{#1}}}

\renewcommand{\mod}{\opname{mod}\nolimits}
\newcommand{\nil}{\opname{nil}\nolimits}

\newcommand{\per}{\opname{per}\nolimits}
\newcommand{\add}{\opname{add}\nolimits}

\newcommand{\Gr}{\opname{Gr}\nolimits}
\newcommand{\dimv}{\underline{\dim}\,}

\newcommand{\tw}{\opname{tw}\nolimits}

\newcommand{\Za}{\opname{Za}\nolimits}

\newcommand{\Z}{\mathbb{Z}}
\newcommand{\N}{\mathbb{N}}
\newcommand{\Q}{\mathbb{Q}}
\newcommand{\C}{\mathbb{C}}

\newcommand{\E}{\mathbb{E}}

\newcommand{\A}{\mathbb{A}}

\newcommand{\la}{\leftarrow}
\newcommand{\iso}{\stackrel{_\sim}{\rightarrow}}
\newcommand{\id}{\mathbf{1}}

\newcommand{\Hom}{\opname{Hom}}

\newcommand{\RHom}{\opname{RHom}}

\newcommand{\Ext}{\opname{Ext}}

\newcommand{\Aut}{\opname{Aut}}
\newcommand{\End}{\opname{End}}

\newcommand{\boxten}{\boxtimes}
\newcommand{\ten}{\otimes}

\newcommand{\ca}{{\mathcal A}}
\newcommand{\cb}{{\mathcal B}}
\newcommand{\cc}{{\mathcal C}}
\newcommand{\cd}{{\mathcal D}}

\newcommand{\cF}{{\mathcal F}}

\newcommand{\ch}{{\mathcal H}}

\newcommand{\cp}{{\mathcal P}}

\newcommand{\ct}{{\mathcal T}}
\newcommand{\cu}{{\mathcal U}}
\newcommand{\cv}{{\mathcal V}}

\newcommand{\eps}{\varepsilon}
\renewcommand{\phi}{\varphi}
\newcommand{\del}{\partial}

\renewcommand{\hat}[1]{\widehat{#1}}

\renewcommand{\tilde}[1]{\widetilde{#1}}

\newcommand{\Frac}{\opname{Frac}}
\newcommand{\Ad}{\opname{Ad}}
\newcommand{\Braid}{\opname{Braid}}

\newcommand{\ccl}{\opname{Ccl}}
\newcommand{\clt}{\opname{Clt}}
\newcommand{\ncc}{\opname{Ncc}}
\newcommand{\trp}{\opname{Trp}}
\renewcommand{\sf}{\opname{Sf}}
\newcommand{\FG}{\opname{FG}}

\date{January 16, 2011, last modified on \today}

\title[On cluster theory and quantum dilogarithm identities]{On cluster theory and 
quantum dilogarithm\\ identities}
\author{Bernhard Keller}

\begin{document}

\begin{abstract}
  These are expanded notes from three survey lectures given at the
  14th International Conference on Representations of Algebras (ICRA XIV)
  held in Tokyo in August 2010. We first study identities between products of 
  quantum dilogarithm series associated with Dynkin quivers following Reineke. 
  We then examine similar identities for quivers with potential and
  link them to Fomin-Zelevinsky's theory of cluster algebras. Here we mainly follow 
  ideas due to Bridgeland, Fock-Goncharov, Kontsevich-Soibelman and Nagao.
\end{abstract}

\begin{classification}
Primary 16G20; Secondary  18E30, 17B37, 14N35.
\end{classification}

\begin{keywords}
Quantum dilogarithm, cluster algebra, Hall algebra,
triangulated category, Calabi-Yau category, Donaldson-Thomas invariant.
\end{keywords}

\maketitle

\section*{Introduction}

The links between the theory of cluster algebras \cite{FominZelevinsky02, FominZelevinsky03,
BerensteinFominZelevinsky05, FominZelevinsky07}
and functional identities for the Rogers dilogarithm first
became apparent through Fomin-Zelevinsky's proof \cite{FominZelevinsky03b} of Zamolodchikov's
periodicity conjecture for $Y$-systems
\cite{Zamolodchikov91} (we refer to
\cite{Nakanishi11} and the references given there for the latest developments
in periodicity in cluster theory and its applications to $T$-systems, $Y$-systems
and dilogarithm identities). This link
was exploited by Fock-Goncharov \cite{FockGoncharov09, FockGoncharov09a},
who emphasize the central r\^ole of the (quantum) dilogarithm both for commutative and for
quantum cluster algebras and varieties.
The quantum dilogarithm is also a key ingredient in Kontsevich-Soibelman's
interpretation \cite{KontsevichSoibelman08} of cluster transformations
in the framework of Donaldson-Thomas theory. Indeed, Kontsevich-Soibelman show that,
under suitable technical hypotheses, if two quivers with potential
are related by a mutation \cite{DerksenWeymanZelevinsky08}, then
their non commutative DT-invariants (\cf section~\ref{ss:intertwiners})
are linked by the composition of a monomial transformation with
the adjoint action of a quantum dilogarithm.
This composition coincides with Fock-Goncharov's cluster transformation
for quantum $Y$-variables \cite{FockGoncharov09a}. Therefore,
Kontsevich-Soibelman's categorical setup for DT-theory contains
a `categorification' of the quantum $Y$-seed mutations and thus,
modulo passage to the `double torus' \cite{FockGoncharov09a} or
to `principal coefficients' \cite{FominZelevinsky07}, of the quantum
$X$-seed mutations. By extending this idea to compositions
of mutations Nagao \cite{Nagao10} has succeeded in deducing
the main theorems on the (additive) categorification
of cluster algebras \cite{DerksenWeymanZelevinsky10} (\cf also
\cite{Plamondon10b}) from Joyce's \cite{Joyce07, Joyce08} and
Joyce-Song's \cite{JoyceSong08} results in DT-theory
(\cf also \cite{Bridgeland10}).

Our aim in the present survey is to give an introduction
to this circle of ideas using quantum dilogarithm identities
as a leitmotif. We start with the classical pentagon
identity (section~\ref{ss:SFK-identity}).
Following Reineke \cite{Reineke03} \cite{Reineke10} and Kontsevich-Soibelman \cite{KontsevichSoibelman08} \cite{Kontsevich10}
we link it to the study of stability functions for a Dynkin quiver of
type $A_2$ and present Reineke's generalization to
arbitrary Dynkin quivers (Theorem~\ref{thm:Reineke}). We sketch
Reineke's beautiful and instructive proof in section~\ref{s:Proof-of-Reinekes-theorem}.
The result can be interpreted as stating that the refined DT-invariant
of a Dynkin (and even an acyclic) quiver is well-defined. In this form, it
is conjectured to generalize to an arbitrary quiver with a potential
with complex coefficients (section~\ref{ss:DT-invar-QP}). Evidence
for this is given in 
Kontsevich-Soibelman's deep work \cite{KontsevichSoibelman08, KontsevichSoibelman10}.
In section~\ref{s:DT-invariants-and-mutations}, we study the behaviour
of the refined DT-invariant under mutations following section~8.4 of
\cite{KontsevichSoibelman08}. We begin by recalling the categorical
setup: The category of finite-dimensional representations of the
Jacobi-algebra of the given quiver with potential is embedded as
the heart of the canonical t-structure in the $3$-Calabi-Yau category $\cd_{fd}\Gamma$,
the full subcategory formed by the homologically finite-dimensional dg modules
over the Ginzburg \cite{Ginzburg06} dg algebra $\Gamma$ associated with the given quiver with potential.
In $\cd_{fd}\Gamma$, the simple representations form a spherical 
collection (section~\ref{ss:setup}) in the sense of \cite{KontsevichSoibelman08}.
Mutation is modeled by `tilting' the heart (section~\ref{ss:comparison-of-categories}),
an idea going back to Bridgeland \cite{Bridgeland05},
cf. also \cite{Bridgeland06} \cite{BridgelandStern10}.

 The comparison formula
for the refined DT-invariants then becomes a simple consequence of
the freedom in the choice of a stability function (section~\ref{ss:comparison-of-invariants}).
The refined DT-invariant is defined to be rational (section~\ref{ss:rational-case})
if its adjoint action is given by a rational transformation. This is the case in
large classes of examples coming from representation theory, Lie theory and
higher Teichm\"uller theory. In this case, the adjoint action
is the `non commutative DT-invariant' \cite{Szendroi08}, whose behaviour under
mutations is governed by Fock-Goncharov's mutation rule for quantum
$Y$-variables (section~\ref{ss:intertwiners}). It is remarkable that
this rule is involutive (section~\ref{ss:mutation-involutive}).
In section~\ref{s:compositions-of-mutations}, we study compositions
of Fock-Goncharov mutations via functors
\[
\FG: \ccl^{op} \to \sf  \quad\mbox{and}\quad \FG: \clt^{op} \to \sf
\]
where $\sf$ is the groupoid of skew fields, $\ccl$ the groupoid
of cluster collections in the ambient $3$-Calabi-Yau category
$\cd_{fd}\Gamma$ and $\clt$ the groupoid of cluster-tilting
sequences in the cluster category $\cc_\Gamma$. The main results
of (quantum) cluster theory may be reformulated by saying
that the image of a morphism $\alpha: S \to S'$ of the
groupoid of cluster collections, where $S$ is the inital
collection, only depends on the target $S'$ (Theorem~\ref{thm:main-thm}).
We define an autoequivalence of $\cd_{fd}\Gamma$ respectively $\cc_\Gamma$
to be reachable if its effect on the inital cluster collection
(respectively cluster tilting sequence) is given by a composition
of mutations. We obtain a homomorphism $F\mapsto \zeta(F)$ from the group of
reachable autoequivalences of the ambient triangulated category to
the group of automorphisms of the functor $\FG$
(sections~\ref{ss:action-autoequivalences-der-cat} and
\ref{ss:action-autoequivalences-cluster-cat}). If the
loop functor $\Omega=\Sigma^{-1}$ of the ambient category
is reachable, then the non commutative DT-invariant
equals $\zeta(\Sigma^{-1})$ (under the assumptions which
ensure that it is well-defined). For example, in the context of the
quivers with potential associated with pairs of Dynkin
diagrams \cite{Keller10a}, the loop functor is reachable
and of finite order, which shows that the non commutative
DT-invariant is of finite order in this case. This fact is of
interest in string theory, \cf section~8 in \cite{CecottiNeitzkeVafa10},
which builds on
\cite{GaiottoMooreNeitzke10a,GaiottoMooreNeitzke09,GaiottoMooreNeitzke10}.
Following
an idea of Nagao \cite{Nagao10} we introduce the groupoid
of nearby cluster collections in section~\ref{ss:nearby-cluster-collections}
and show how it can be used to prove Theorem~\ref{thm:main-thm}
(section~\ref{ss:proof-main-thm}) and to reconstruct
the refined DT-invariant (Theorem~\ref{thm:reconstruct-refined}).
We conclude by presenting a purely combinatorial realization
of the groupoid of nearby cluster collections: the tropical
groupoid. In many cases, it allows for a purely combinatorial
construction of the refined DT-invariant (section~\ref{ss:tropical-groupoid}).

\section*{Acknowledgment} I thank Sarah Scherotzke for providing me
with her notes of the lectures and Tom Bridgeland, Changjian Fu, 
Christof Gei\ss, Pedro Nicol\'as, 
Pierre-Guy Plamondon and especially Mike Gorsky
for their remarks on a preliminary version of this article.
 I am grateful to Kentaro Nagao, So Okada and Cumrun Vafa for
stimulating (email) conversations and to Maxim Kontsevich for many inspiring
lectures on the subject.


\section{Quantum dilogarithm identities from Dynkin quivers, after Reineke}
\label{s:Quantum-dilog-identities-Dynkin}

\subsection{The pentagon identity} \label{ss:SFK-identity}
Let $q^{1/2}$ be an indeterminate. We denote its square by $q$. The
{\em quantum dilogarithm} is the (logarithm of the) series
\begin{equation}
\E(y) = 1 + \frac{q^{1/2}}{q-1}\cdot y + \cdots  + \frac{q^{n^2/2} y^n}{(q^n-1)(q^n-q)\cdots (q^n-q^{n-1})}+ \cdots
\end{equation}
considered as an element of the power series algebra $\Q(q^{1/2})[[y]]$. Notice that the denominator
of its general coefficient is the polynomial which computes the order of the general
linear group over a finite field with $q$ elements. Let us define the
{\em quantum exponential} by
\[
\exp_q(y)= \sum_{n=0}^\infty \frac{y^n}{[n]!} \ko
\]
where $[n]!$ is the polynomial which computes the number of complete flags
in an $n$-dimensional vector space over a field with $q$ elements. Then we have
\begin{equation} \label{eq:E-series}
\E(y)=\exp_q(\frac{q^{1/2}}{q-1}\cdot y) \ko
\end{equation}
which explains the choice of the notation $\E$ (for the mysterious
scaling factor, \cf Remark~\ref{ss:remark-scaling-factor}).
The quantum dilogarithm has many
remarkable properties (\cf \eg \cite{Zagier07} and the references given there),
among which we single out the following.

\begin{theorem}[Sch\"utzenberger \cite{Schuetzenberger53}, 
Faddeev-Volkov \cite{FaddeevVolkov93}, Faddeev-Kashaev \cite{FaddeevKashaev94}] For two indeterminates
$y_1$ and $y_2$ which $q$-commute in the sense that
\[
y_1 y_2=q y_2 y_1 \ko
\]
we have the equality
\begin{equation} \label{eq:SFK-identity}
\E(y_1) \E(y_2) = \E(y_2) \E(q^{-1/2} y_1 y_2) \E(y_1).
\end{equation}
\end{theorem}

As shown in \cite{FaddeevKashaev94}, this equality implies the classical `pentagon identity' for
Rogers' dilogarithm.

\subsection{Reineke's identities} \label{ss:Reineke-s-identities}
Our aim in this section is to associate, following Reineke \cite{Reineke10}, an identity
analogous to (\ref{eq:SFK-identity}) with each simply laced Dynkin diagram so that the
above identity corresponds to the diagram $A_2$.

So let $\Delta$ be a simply laced Dynkin diagram and let $Q$
be a quiver (=oriented graph) with underlying graph $\Delta$.
We will
associate a whole family of quantum dilogarithm products with $Q$. All these products
will be equal and among the resulting equalities, we will obtain the required generalization
of the pentagon identity (\ref{eq:SFK-identity}). The quantum dilogarithm products
will be constructed from `stability functions' on the category of representations
of $Q$.

We first need to introduce some notation: Let $k$ be a field. Let $\ca$ be
the category of representations of the opposite quiver $Q^{op}$ with values in the
category of finite-dimensional $k$-vector spaces
(we refer to \cite{Ringel84} \cite{GabrielRoiter92} \cite{AuslanderReitenSmaloe95}
\cite{AssemSimsonSkowronski06} for quiver representations).
Let $I=\{1, \ldots, n\}$ be the set of vertices of $Q$.
For each vertex $i$, let $S_i$ be the simple representation whose value at
$i$ is $k$ and whose value at all other vertices is zero. Let $K_0(\ca)$ be
the Grothendieck group of $\ca$. It is a finitely generated free abelian group and admits the
classes of the representations $S_i$, $i\in I$, as a basis.

A {\em stability function on $\ca$} is a group homomorphism
\[
Z: K_0(\ca) \to \C
\]
to the underlying abelian group of the field of complex numbers such
that for each non zero object $X$ of $\ca$, the number $Z(X)$ is non
zero and its argument, called the {\em phase of $X$}, lies in the
interval $[0,\pi[$, cf.~figure~\ref{fig:central-charge}.
A non zero object $X$ of $\ca$ is {\em semi-stable}
(respectively {\em stable}) if for each non zero proper subobject $Y$ of
$X$, the phase of $Y$ is less than or equal to the phase of $X$ (respectively
strictly less than the phas of $X$). Sometimes, the homomorphism $Z$
is called a {\em central charge}. Since it is a group homomorphism,
it is determined by the complex numbers $Z(S_i)$, $i\in I$. Notice
that each simple object of $\ca$ is stable (since it has no non
zero proper subobjects).

\begin{proposition}[King \cite{King94}] \label{prop:King} Let $Z:K_0(\ca)\to \C$ be a
stability function. For each real number $\mu$, let $\ca_\mu$ be
the full subcategory of $\ca$ whose objects are the zero object and the
semistable objects of phase $\mu$.
\begin{itemize}
\item[a)] The subcategory $\ca_\mu$ of $\ca$ is stable under
forming extensions, kernels and cokernels in $\ca$. In particular,
it is abelian and its inclusion into $\ca$ is exact. Its simple
objects are precisely the stable objects of phase~$\mu$.
\item[b)] Each object $X$ admits a unique filtration
\[
0=X_0 \subset X_1 \subset \cdots \subset X_s=X
\]
whose subquotients are semistable with strictly decreasing phases.
It is called the {\em Harder-Narasimhan filtration} (or {\em HN-filtration}) of $X$.
Its subquotients are the {\em HN-subquotients} of $X$.
\end{itemize}
\end{proposition}

Notice that by part a), all stable objects are indecomposable
and their endomorphism algebras are (skew) fields.

Under the assumptions of the proposition, let $\ca_{\geq \mu}$ be
the full subcategory of $\ca$ formed by the objects $X$ all of
whose HN-subquotients have phase greater than or equal to $\mu$.
Define the full subcategory $\ca_{<\mu}$ analogously. Then
the pair $(\ct, \cF)=(\ca_{\geq \mu}, \ca_{<\mu})$ is a {\em torsion pair}
in $\ca$, i.e. for $X$ in $\ct$ and $Y$ in $\cF$,
we have $\Hom(X,Y)=0$ and for each object $Z$, there is an
exact sequence
\[
\xymatrix{ 0 \ar[r] & X \ar[r] & Z \ar[r] & Y \ar[r] & 0 } \ko
\]
where $X$ belongs to $\ct$ and $Y$ to $\cF$.
Thus, each stability function $Z$ determines a decreasing chain
(indexed by the real numbers) of torsion subcategories $\ca_{\geq \mu}$.
Below, instead of working with stability functions, we could work
more generally with decreasing chains of torsion subcategories (but
stability functions are more pleasant to use).

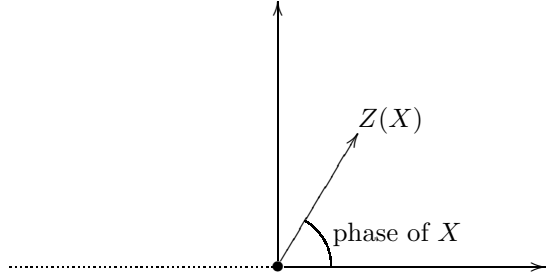
\begin{figure}
\[
\xyoption{arc}
\begin{xy}
 0;<1pt,0pt>:<0pt,1pt>::
(25,12)*[r]{\mbox{phase of $X$}},
(0,0) *{\bullet} = "0",
(100,0) *{} = "A",
(-100,0) *{} ="B",
(0,100) *{} ="C",
(30,50) *{} = "D",
 "0", {\ar "A"},
"0", "B" **@{.},
"0", {\ar "D"},
"D" *!DL{Z(X)} ,
"0", {\ar "C"},
(20,0); (10.28,17.14)  \ellipse^{-}
\end{xy}
\]
\caption{Image of a non zero object under the central charge} \label{fig:central-charge}
\end{figure}

As a first example, let us consider the case where $Q$ is the
quiver
\[
\xymatrix{ 1 \ar[r] & 2}
\]
of type $A_2$. Its Auslander-Reiten quiver is
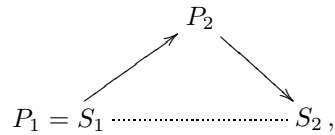
\[
\xymatrix{  & P_2 \ar[dr] & \\
P_1=S_1 \ar[ru] \ar@{.}[rr] & & S_2\, ,} 
\]
where $S_1$ and $S_2$ are the simple representations associated
with the vertices and $P_2$ the projective cover of $S_2$ given
by the identity map $k \la k$ (we consider representations of $Q^{op}$).
In particular, the representation $P_2$ is the only non simple indecomposable
object and it has $S_1$ as its unique non zero proper subobject.
If we only consider {\em generic} stability functions
(i.e. the central charge $Z$ maps two non zero classes into the same straight
line of $\C$ only if they are $\Q$-proportional), there are essentially
two possibilities: Either the phase of $S_1$ is greater than that of $S_2$ or
the phase of $S_1$ is smaller than that of $S_2$, cf.~figure~\ref{fig:stab-A2}.
\begin{figure}
\[
\begin{array}{ccc}
\xyoption{arc}
\begin{xy}
 0;<0.9pt,0pt>:<0pt,0.9pt>::
(0,0) *{\bullet} = "0",
(100,0) *{} = "A",
(-50,0) *{} ="B",
(0,100) *{} ="C",
(-20,30) *+!DR{Z(S_1)} = "S1", *{\bullet},
(50, 40) *+!DL{Z(S_2)} = "S2", *{\bullet},
(30, 70) *+!D{Z(P_2)}= "P2", *{\circ},
 "0", {\ar "A"},
"0", "B" **@{.},
"0", {\ar "C"},
"0", {\ar "S1"},
"0", {\ar@{.>} "P2"},
"0", {\ar "S2"}
\end{xy}
& &
\xyoption{arc}
\begin{xy}
 0;<0.9pt,0pt>:<0pt,0.9pt>::
(0,0) *{\bullet} = "0",
(100,0) *{} = "A",
(-50,0) *{} ="B",
(0,100) *{} ="C",
(-40,40) *+!DR{Z(S_2)} = "S2", *{\bullet},
(70, 40) *+!DL{Z(S_1)} = "S1", *{\bullet},
(30, 80) *+!D{Z(P_2)}= "P2", *{\bullet},
 "0", {\ar "A"},
"0", "B" **@{.},
"0", {\ar "C"},
"0", {\ar "S1"},
"0", {\ar "P2"},
"0", {\ar "S2"}
\end{xy}
\end{array}
\]
\caption{Two different generic stability functions for $A_2$} \label{fig:stab-A2}
\end{figure}
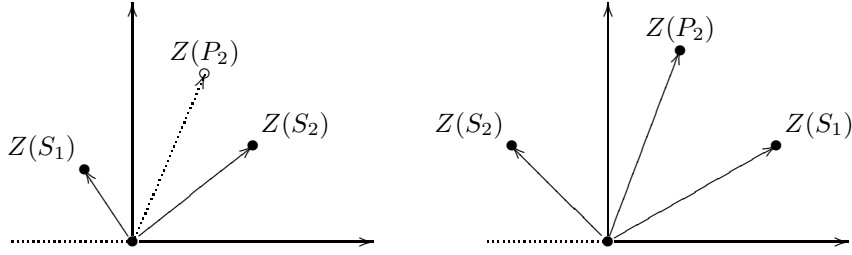
Now the
object $P_2$, whose class in $K_0(\ca)$ is the sum of those of $S_1$ and $S_2$,
has the proper subobject $S_1$. In the first case, $S_1$  is of greater phase than
$P_2$ and so $P_2$ is unstable. In the second case, $P_2$ is stable, since its only
non zero proper subobject $S_1$ has smaller phase. Thus, in the
first case, we find {\em two} stable objects and in the second
case {\em three}. Notice that these numbers agree with the numbers
of factors in the two products appearing in the pentagon identity~(\ref{eq:SFK-identity}).
In both cases, each semistable object of given phase $\mu$ is a direct
sum of copies of the unique stable object of phase $\mu$. Thus,
these stability functions are discrete in the sense of the
following definition.

\begin{definition} A stability function $K_0(\ca)\to \C$ is {\em discrete} if,
for each real number $\mu$,  the subcategory of semistable objects $\ca_\mu$
is zero or semisimple with a unique simple object.
\end{definition}

We will associate a product of quantum dilogarithms with each discrete
stability function. We first need to define the algebra in
which these products will be computed:
Let $n\geq 1$ be an integer and $Q$ a finite quiver whose vertex
set is the set of integers from $1$ to $n$. For a
$kQ$-module $M$, let the dimension vector $\dimv M \in\Z^n$
have the components $\dim(Me_i)$. The map $\dimv$ induces
an isomorphism from $K_0(\ca)$ onto $\Z^n$. Define the
{\em Euler form} $\langle, \rangle: \Z^n\times \Z^n \to \Z$ by
\[
\langle \dimv L , \dimv M\rangle = \dim \Hom(L,M) - \dim \Ext^1(L,M).
\]
Define $\lambda: \Z^n \times \Z^n \to \Z$ to be opposite to
the antisymmetrization of the Euler form so that we have
\[
\lambda(\alpha,\beta) = \langle\beta,\alpha\rangle-\langle\alpha,\beta\rangle.
\]
for all $\alpha$, $\beta$ in $\Z^n$. Notice that $\lambda(e_i,e_j)$
is the difference of the number of arrows from $i$ to $j$ minus the
number of arrows from $j$ to $i$.
The {\em quantum affine space $\A_Q$} is the $\Q(q^{1/2})$-algebra
generated by the variables $y^\alpha$, $\alpha\in\N^n$, subject to
the relations
\[
y^\alpha y^\beta = q^{1/2 \,\lambda(\alpha,\beta)} y^{\alpha+\beta}
\]
for all $\alpha$ and $\beta$ in $\N^n$. It is also generated by
the variables $y_i = y^{e_i}$, $1\leq i\leq n$, subject to the
relations
\[
y_i y_j = q^{\lambda(e_i, e_j)} y_j y_i
\]
and admits the monomials $y^\alpha$, $\alpha\in\N^n$, as a basis
over $\Q(q^{1/2})$. The {\em formal quantum affine space $\hat{\A}_Q$}
is the completion of $\A_Q$ with respect to the ideal generated by the
$y_i$, $i\in I$.

\begin{theorem}[Reineke \cite{Reineke10}] \label{thm:Reineke} Let $k$ be a field,
$Q$ a Dynkin quiver and
\[
Z: K_0(\mod kQ) \to \C
\]
a discrete stability function. Then the product computed in $\hat{\A}_Q$
\begin{equation} \label{eq:Reineke}
\E_{Q,Z} = \prod^\curvearrowright_{M\mbox{\scriptsize stable}} \E(y^{\dimv M}) \ko
\end{equation}
where the factors are in the order of decreasing phase,
is independent of the choice of $Z$.
\end{theorem}

Recall that if $Q$ is a quiver without oriented cycles, a {\em source}
of $Q$ is a vertex without incoming arrows; a {\em source sequence}
for $Q$ is an enumeration $i_1$, \ldots, $i_n$ of the
vertices of $Q$ such that each $i_j$ is a source in the quiver obtained
from $Q$ by removing the vertices $i_1$, \ldots, $i_{j-1}$ and
all arrows incident with one of these vertices. If $Q$ is a Dynkin
quiver, by Gabriel's theorem, each positive root $\alpha$ of
the corresponding root system is the dimension vector of
an indecomposable representation $V(\alpha)$, unique up
to isomorphism. We endow the set of positive roots with the
smallest order relation such that $\Hom(V(\alpha), V(\beta))\neq 0$
implies $\alpha \leq \beta$.
\begin{corollary}
If $Q$ is a Dynkin quiver, we have
\[
\E(y_{i_1}) \ldots \E(y_{i_n}) = \E(y^{\alpha_1}) \ldots \E(y^{\alpha_N}) \ko
\]
where $i_1$, \ldots, $i_n$ is a source sequence for $Q$ and
$\alpha_1$, \ldots, $\alpha_N$ are the dimension vectors
of the indecomposable representations enumerated in decreasing
order with respect to the above defined order relation\end{corollary}

Clearly, for the quiver $Q:1\to 2$, the corollary specializes to
Sch\"utzenberger-Faddeev-Kashaev's theorem. In this case, we deduce it from the
theorem by comparing $\E_{Q,Z}$ for the two generic stability functions
considered above. To deduce the corollary from the theorem in the
general case, we need to construct stability functions $Z_1$ and $Z_2$ such that
\begin{itemize}
\item for $Z_1$, the only stable objects are the simples and $S_{i_1}$, \ldots,
$S_{i_n}$ have strictly decreasing phases;
\item for $Z_2$, the stable objects are precisely the indecomposable
representations and their phases increase from left to right
in the Auslander-Reiten quiver.
\end{itemize}
The existence of $Z_1$ is not hard to check. The existence of $Z_2$
was a conjecture by Reineke \cite{Reineke10} proved by Hille-Juteau 
(unpublished). It would be interesting to try and apply the methods of 
section~2.6 of \cite{IngallsThomas09} to this problem.
In \cite{Reineke10}, the corollary is proved using only the
HN-filtrations associated with $Z_2$, which are easy to construct directly without
assuming the existence of $Z_2$ itself.

\subsection{The Kronecker quiver} \label{ss:Kronecker} Let us emphasize
that the main thrust of Reineke's work in \cite{Reineke10} and
Kontsevich--Soibelman's work in \cite{KontsevichSoibelman08} and
\cite{KontsevichSoibelman10} bears on the general case of
not necessarily discrete stability functions. As an example, consider
the Kronecker quiver
\[
\xymatrix{1 \ar@<1ex>[r] \ar@<-1ex>[r] & 2}.
\]
For a stability function such that the phase of $S_1$
is strictly greater than the one of $S_2$,
the simples are the only stable objects, which leads to the product
\begin{equation} \label{eq:kr-left}
\E(y_1) \E(y_2)
\end{equation}
For a stability function such that the phase of $S_1$ is strictly
smaller than the one of $S_2$, the stable representations are precisely
the postprojective indecomposables, the preinjective indecomposables
and the representations in the $\mathbb{P}^1$-family
\[
\xymatrix{\C  & \C \ar@<1ex>[l]^{x_0} \ar@<-1ex>[l]_{x_1}} \ko
(x_0:x_1)\in \mathbb{P}^1(\C).
\]
In this case, Reineke and
Kontsevich--Soibelman obtain the product (\cf equation~(2.24) of
\cite{GaiottoMooreNeitzke10a})
\begin{equation} \label{eq:kr-right}
\big(\E(0,1) \E(1,2) \E(2,3) \ldots \big)\E(1,1)^4 \E(2,2)^{-2} \big( \ldots
\E(3,2)\E(2,1)\E(1,0) \big) \ko
\end{equation}
where we write $\E(a,b)$ instead of $\E(y^{a e_1 + b e_2})$. An elementary
proof of the identity between the expressions (\ref{eq:kr-left}) and
(\ref{eq:kr-right}) can be found in Appendix~A of
\cite{GaiottoMooreNeitzke10a}.

\section{Sketch of proof of Reineke's theorem}
\label{s:Proof-of-Reinekes-theorem}

Let us fix a discrete stability condition $Z: K_0(\mod kQ) \to \C$.
Since $Q$ is a Dynkin quiver, its representation theory is `independent of $k$'.
In particular, the stability function $Z$ is discrete
for any choice of $k$ and the dimension vectors of the stable objects
do not depend on $k$. So the product~(\ref{eq:Reineke})
is independent of $k$.

Now notice that the coefficients of the series $\E_{Q,Z}$ in
(\ref{eq:Reineke}) lie in the subring $R$ of $\Q(q^{1/2})$
obtained from the polynomial ring $\Q[q^{1/2}]$ by
inverting $q^{1/2}$ and all the polynomials $q^l-1$,
$l\geq 1$. Thus, in order to prove that $\E_{Q,Z}$
is independent of $Z$, it suffices to show that
its image under specializing $q$ to any prime
power $p^m$ is independent of $Z$.

So let us assume now that $q$ is a prime power $p^m$ and that
$k$ is a finite field with $q$ elements. Let $\ca$ be
the category of finite-dimensional $kQ$-modules.
We consider the completed (non twisted, opposite) Ringel-Hall algebra
$\hat{\ch}(\ca)$: its elements are formal series with rational
coefficients
\[
\sum_{[\ca]} a_M [M] \ko
\]
where the sum is taken over the set $[\ca]$ of
isomorphism classes $[M]$ of $k$-finite-dimensional
right $kQ$-modules $M$. The product of $\hat{\ch}(\ca)$ is the
continuous bilinear map determined by
\[
[L] [M] = \sum c_{LM}^N [N] \ko
\]
where $c_{LM}^N$ is the number of submodules
$L'$ of $N$ which are isomorphic to $L$ and such
that $N/L'$ is isomorphic to $M$.

By King's Proposition~\ref{prop:King}, each $kQ$-module
$M$ admits a unique HN-filtration with semistable subquotients
of strictly decreasing phase. In the Ringel-Hall algebra,
this translates into the identity
\begin{equation} \label{eq:Hall-algebra-identity}
\sum_{[\ca]} [M] = \prod^\curvearrowright_{\mu\mbox{\scriptsize\  decreasing}} \sum_{[\ca_\mu]} [L] \ko
\end{equation}
where the sum on the left is taken over all isomorphism classes,
the product on the right over the possible (rational) phases
in decreasing order and each factor is the sum over the
isomorphism classes of semi-stable modules $L$ with phase $\mu$.
This identity shows that the product on the right hand
side is in fact independent of the choice of $Z$.
It remains to `transport' this identity from the Hall
algebra to our algebra of non commutative power series.
For this, we define $\hat{\A}_{Q, \mbox{\scriptsize spec}}$ to be
the algebra obtained by specializing
$q=p^m$ in the subring of $\hat{\A}_Q$ formed by the series with
coefficients in the ring $R$ defined above and
we define the {\em integration map}
\[
\int: \hat{\ch}(\ca) \to \hat{\A}_{Q, \mbox{\scriptsize spec}}
\]
to be the continuous $\Q$-linear map which takes the class $[M]$ of a module to
\[
q^{1/2 \langle\dimv M, \dimv M \rangle}\frac{y^{\dimv M}}{|\Aut(M)|}\; \ko
\]
where $|\Aut(M)|$ is the order of the automorphism group of $M$, the
notation $\dimv M$ denotes the class of $M$ in $K_0(\ca)$ and the form
$\langle , \rangle$ is the {\em Euler form} defined by
\[
\langle \dimv L, \dimv M \rangle = \dim \Hom(L,M) - \dim\Ext^1(L,M).
\]

\begin{lemma} The integration map is an algebra homomorphism.
\end{lemma}

For a proof, see for example Lemma~1.7 of \cite{Schiffmann06}
or Lemma~6.1 of \cite{Reineke03}.
One easily computes that if $M$ is a module with $\End(M)=k$,
then we have
\[
\int \sum [M^n] = \E(y^{\dimv M}).
\]
Hence if we apply the integration map to the identity~(\ref{eq:Hall-algebra-identity})
and use the fact that all objects of $\ca_\mu$ are sums of
a unique stable object, we find the equality
\[
\int \sum_{[\ca]} [L] = \prod^{\curvearrowright}_{M\mbox{\scriptsize stable}} \E(y^{\dimv M}) \ko
\]
where the factors on the right appear in the order of decreasing phase.
This equality clearly shows that the right hand side does not depend
on the choice of $Z$.

\section{Quantum dilogarithm identities from quivers with potential, after Kontsevich-Soibelman}
\label{s:Quantum-Dilog-identities-QP}

\subsection{DT-invariants for quivers with potential} \label{ss:DT-invar-QP}
Let $Q$ be a finite quiver and $k$ a field. Let $kQ$ be the path algebra
of $Q$ and $\hat{kQ}$ its completion with respect to path length. Thus,
the paths in $Q$ form a topological basis of $\hat{kQ}$.
The continuous zeroth Hochschild homology of the completed path algebra
\[
HH_0(\hat{kQ})
\]
is the completion of the quotient of the topological linear space $\hat{kQ}$
by the subspace $[\hat{kQ},\hat{kQ}]$ of all commutators. It has a topological
basis formed by the classes (modulo cyclic permutation) of cyclic paths of $Q$.
For each arrow $\alpha$ of $Q$, we have the {\em cyclic derivative}
\[
\del_\alpha: HH_0(\hat{kQ}) \to \hat{kQ}
\]
which is the continuous linear map taking the equivalence class
of a path $p$ to the sum
\[
\sum vu
\]
taken over all decompositions $p=u\alpha v$.
A {\em potential on $Q$} is an element $W$ of $HH_0(\hat{kQ})$
which does not involve cycles of length $\leq 2$. A potential
is {\em polynomial} if it is linear combination of finitely many cycles.
The {\em Jacobian algebra} $\cp(Q,W)$ is
the quotient of $\hat{kQ}$ by the twosided ideal generated by
the cyclic derivatives $\del_\alpha W$, where $\alpha$ runs through
the arrows of $Q$. We define
\[
\nil(\cp(Q,W))
\]
to be the category of finite-dimensional right $\cp(Q,W)$-modules
(such a module is automatically nilpotent, \ie each element is annihilated by
all long enough paths). Clearly, this is an abelian category where each object
has finite length and whose simples are the modules $S_i$ associated
with the vertices $i$ of $Q$. The author thanks M.~Kontsevich
for informing him \cite{Kontsevich11} that the following statement is still
conjectural.

\begin{conjecture} \label{conj:KS} Suppose that $k=\C$ and
that $W$ is a polynomial potential. Suppose that
we have a discrete stability function
\[
Z: K_0(\nil(\cp(Q,W))) \to \C.
\]
Then the product
\[
\E_{Q,W,Z} = \prod^\curvearrowright_{M \mbox{\scriptsize\ stable}} \E(y^{\dimv M}) \ko
\]
where the factors appear in the order of decreasing phase, is independent
of the choice of $Z$.
\end{conjecture}

Evidence for the conjecture comes from Kontsevich-Soibelman's 
deep work in
\cite{KontsevichSoibelman08} \cite{KontsevichSoibelman10}. With a
(much) better definition of $\E_{Q,W,Z}$, it generalizes to arbitrary
stability functions $Z$ and is in fact expected to hold for not necessarily
polynomial potentials. The strategy one would like to adopt for the proof is similar to
Reineke's but
\begin{itemize}
\item we have to work over the complex numbers and have to
replace the Hall algebra by the `motivic Hall algebra',
\cf \cite{Joyce06, Joyce07, Joyce07a, JoyceSong08} as well as
\cite{KontsevichSoibelman08} \cite{Bridgeland10},
\item the existence of the integration map remains conjectural
(it is stated as a theorem in section~6.3 of \cite{KontsevichSoibelman08} and
an important ingredient for its still incomplete proof is the integral
identity studied in section~7.8 of \cite{KontsevichSoibelman10}).
\end{itemize}
For the quivers with potential $(Q,W)$ satisfying the claim of
the conjecture, the {\em refined DT-invariant}
is defined as
\begin{equation} \label{eq:refined-DT-invariant}
\E_{Q,W} = \E_{Q,W,Z} \ko
\end{equation}
where $Z$ is any discrete stability function.

\subsection{Remark on the scaling factor} \label{ss:remark-scaling-factor} In fact,
Kontsevich--Soibelman consider the motivic Hall algebra
not of the abelian category $\nil\cp(Q,W)$ but of the triangulated
$3$-Calabi-Yau category $\cd_{fd} \Gamma(Q,W)$ defined in section~\ref{ss:setup}
below. This category contains $\nil\cp(Q,W)$ as the heart of the natural $t$-structure.
The structure of this larger Hall algebra is unknown except in the
case where the quiver $Q$ has one vertex (labeled $1$) and no arrows.
In this case, the Hall algebra $\ch$ of the corresponding category defined
over a finite field is described explicitly in
\cite{KellerYangZhou09}. The description shows that the largest commutative
quotient of $\ch$ is the algebra
\begin{equation} \label{eq:com-quot}
\Q(q^{1/2})[z_i\; | \; i\in\Z]/\big(z_{i+1} z_i  = \frac{q}{(q-1)^2}\big) \ko
\end{equation}
where the generator $z_i$ is the image of the shifted
simple module $\Sigma^{-i}S_1$, \cf Theorem~5.1 and Lemma~6.1 of
[loc. cit.]. Now notice that for the quiver $Q$ (which does not have arrows),
the algebra $\A_Q$ is commutative. Let us denote by $\Sigma$ its
automorphism mapping the generator $y_1$ to $y_1^{-1}$.
Using~(\ref{eq:com-quot}) we see  that there
is a unique $\Q(q^{1/2})$-algebra homomorphism
\[
\int : \ch \to \A_Q
\]
such that we have $\int \circ \Sigma = \Sigma \circ \int$ and that the image of
$z_0$ is a scalar multiple of $y_1$. We clearly have
\[
\int z_0 = \frac{q^{1/2}}{q-1} \, y_1 \, .
\]
This explains the choice of the scaling factor in the definition
of $\E(y)$ in equation~(\ref{eq:E-series}).

\section{DT-invariants and mutations}
\label{s:DT-invariants-and-mutations}

\subsection{The comparison problem}
Let $(Q,W)$ be a quiver with a polynomial potential as in
section~\ref{s:Quantum-Dilog-identities-QP} and assume
that conjecture~\ref{conj:KS} holds for $(Q,W)$ so that the
refined DT-invariant $\E_{Q,W}$ is well-defined. Let $k$
be a vertex of $Q$ not lying on a loop or a $2$-cycle.
Then the mutated quiver with potential $(Q',W')=\mu_k(Q,W)$
is well-defined, \cf \cite{DerksenWeymanZelevinsky08}.
Let us assume that $W'$ is again polynomial and that
conjecture~\ref{conj:KS} holds for $(Q',W')$ as well. Then, according
to section~\ref{s:Quantum-Dilog-identities-QP}, we have well-defined
refined DT-invariants
\[
\E_{Q,W} \in \hat{\A}_Q \mbox{ and } \E_{Q',W'} \in \hat{\A}_{Q'}.
\]
We wish to compare them. For this, we need to recall the links
between the abelian categories $\ca$ and $\ca'$ of nilpotent
modules over $\cp(Q,W)$ respectively $\cp(Q',W')$.
We do this without supposing that $W$ or $W'$ are polynomial.
We mainly follow section~8 of \cite{KontsevichSoibelman08}
and \cite{Nagao10}.

\subsection{The setup} \label{ss:setup} Let $Q$ be a finite
quiver and $W$ a potential on $Q$ (\cf section~\ref{ss:DT-invar-QP}).
Let $\Gamma$ be the Ginzburg \cite{Ginzburg06} dg algebra
of $(Q,W)$. It is constructed as follows: Let $\tilde{Q}$ be the
graded quiver with the same vertices as $Q$ and whose arrows are
\begin{itemize}
\item the arrows of $Q$ (they all have degree~$0$),
\item an arrow $a^*: j \to i$ of degree $-1$ for each arrow $a:i\to j$ of $Q$,
\item a loop $t_i : i \to i$ of degree $-2$ for each vertex $i$
of $Q$.
\end{itemize}
The underlying graded algebra of $\Gamma(Q,W)$ is the
completion of the graded path algebra $k\tilde{Q}$ in the category
of graded vector spaces with respect to the ideal generated by the
arrows of $\tilde{Q}$. Thus, the $n$-th component of
$\Gamma(Q,W)$ consists of elements of the form
$\sum_{p}\lambda_p p$, where $p$ runs over all paths of degree $n$.
The differential of $\Gamma(Q,W)$ is the unique continuous
linear endomorphism homogeneous of degree~$1$ which satisfies the
Leibniz rule
\[
d(uv)= (du) v + (-1)^p u dv \ko
\]
for all homogeneous $u$ of degree $p$ and all $v$, and takes the
following values on the arrows of $\tilde{Q}$:
\begin{itemize}
\item $da=0$ for each arrow $a$ of $Q$,
\item $d(a^*) = \del_a W$ for each arrow $a$ of $Q$,
\item $d(t_i) = e_i (\sum_{a} [a,a^*]) e_i$ for each vertex $i$ of $Q$, where
$e_i$ is the lazy path at $i$ and the sum runs over the set of
arrows of $Q$.
\end{itemize}
The Ginzburg algebra should be viewed as a refined version of the
Jacobian algebra $\cp(Q,W)$. It is concentrated in (cohomological) degrees $\leq 0$
and $H^0(\Gamma)$ is isomorphic to $\cp(Q,W)$.
We refer to \cite{KellerYang11} for more details on the setup which we now describe.
Let $\cd(\Gamma)$ be the derived category of $\Gamma$,
$\per(\Gamma)$ the perfect derived category and $\cd_{fd}\Gamma$ the
full subcategory of $\cd(\Gamma)$ formed by the dg modules
whose homology is of finite total dimension. The category
$\cd_{fd}\Gamma$ is in fact contained in $\per(\Gamma)$. It
is triangulated, has finite-dimensional morphism spaces (even
its graded morphism spaces are of finite total dimension)
and is $3$-Calabi-Yau, by which we mean that we have bifunctorial
isomorphisms
\[
D\Hom(X,Y) \iso \Hom(Y,\Sigma^3 X) \ko
\]
where $D$ is the duality functor $\Hom_k(?,k)$ and $\Sigma$ the
shift functor. The simple $\cp(Q,W)$-modules $S_i$ can be viewed
as $\Gamma$-modules via the canonical morphism $\Gamma \to H^0(\Gamma)$.
They then become $3$-spherical objects in $\cd_{fd}\Gamma$. They
yield the Seidel-Thomas \cite{SeidelThomas01} twist functors $\tw_{S_i}$.
These are autoequivalences
of $\cd\Gamma$ such that each object $X$ fits into a triangle
\[
\RHom(S_i,X) \ten_k S_i \to X \to \tw_{S_i}(X) \to \Sigma \RHom(S_i,X) \ten_k S_i\;.
\]
By \cite{SeidelThomas01}, the twist functors give rise to a (weak) action on $\cd(\Gamma)$ of the
braid group associated with $Q$, i.e.~the group with generators
$\sigma_i$, $i\in Q_0$, and relations $\sigma_i \sigma_j = \sigma_j \sigma_i$
if $i$ and $j$ are not linked by an arrow and
\[
\sigma_i \sigma_j \sigma_i = \sigma_j \sigma_i \sigma_j
\]
if there is exactly one arrow between $i$ and $j$ (no relation
if there are two or more arrows).

The category $\cd(\Gamma)$ admits a natural $t$-structure
whose truncation functors are those of the natural $t$-structure
on the category of complexes of vector spaces (because $\Gamma$
is concentrated in degrees $\leq 0$). Thus, we have an
induced natural $t$-structure on $\cd_{fd}\Gamma$. Its
heart $\ca$ is canonically equivalent to the category
$\nil(\cp(Q,W))$. In particular, the inclusion of $\ca$
into $\cd_{fd}\Gamma$ induces an isomorphism in the
Grothendieck groups
\[
K_0(\ca) \iso K_0(\cd_{fd}\Gamma).
\]
Notice that the lattice $K_0(\cd_{fd}\Gamma)$ carries the
canonical form defined by
\[
\lambda(X, Y) = \sum_{p\in\Z} (-1)^p\dim \Hom(X,\Sigma^p Y).
\]
It is skew-symmetric thanks to the $3$-Calabi-Yau property. It
also follows from the Calabi-Yau property and from the fact
that $Ext^i_\ca(L,M)=\Ext^i(L,M)$ for $i=0$ and $i=1$ (but
not $i>1$ in general!) that for two objects $L$ and $M$ of $\ca$,
we have
\[
\lambda(L,M) = \dim \Hom(L,M) - \dim \Ext^1(L,M) + \dim\Ext^1(M,L) - \dim\Hom(M,L).
\]
Since the dimension of $\Ext^1(S_i,S_j)$ equals the number of
arrows in $Q$ from $j$ to $i$, we obtain that the matrix of $\lambda$
in the basis of the simples of $\ca$ has its $(i,j)$-coefficient equal
to the number of arrows from $i$ to $j$ minus the number of arrows
from $j$ to $i$ in $Q$.

Suppose that $\Lambda$ is a lattice endowed with a skew-symmetric bilinear form
$\lambda: \Lambda \times \Lambda \to \Z$.
Let $C\subset \Lambda$ be a cone (a subset containing $0$ and stable under
forming sums). Then we define the $\Q(q^{1/2})$-algebra $\A_C$ to be generated by the symbols $y^\alpha$,
$\alpha\in C$, subject to the relations
\[
y^\alpha y^\beta = q^{1/2 \cdot \lambda(\alpha,\beta)} y^{\alpha+\beta}.
\]
If $C$ does not contain $-\alpha$ for any non zero $\alpha$ in $C$,
we define $\hat{\A}_C$ to be the completion of $\A_C$ with respect to
the ideal generated by the $y^\alpha$, $\alpha\neq 0$.

Via the identification $y_i = y^{[S_i]}$, we can now interpret the 
algebra $\hat{\A}_Q$ of section~\ref{s:Quantum-dilog-identities-Dynkin}
intrinsically as the algebra $\hat{\A}_{K_0^+(\ca)}$ associated with the
cone $K^+_0(\ca)$ of positive elements in the Grothendieck group $K_0(\ca)$,
which we endow with the form induced from that of $K_0(\cd_{fd}\Gamma)$
via the canonical isomorphism $K_0(\ca)\iso K_0(\cd_{fd}\Gamma)$.

\subsection{Comparison of categories} \label{ss:comparison-of-categories}
Keep the notations from the preceding section. In addition, let
$k$ be a vertex of $Q$ not lying on a $2$-cycle and let
$(Q',W')$ be the mutation of $(Q,W)$ at $k$ in the sense
of \cite{DerksenWeymanZelevinsky08}. Let $\Gamma'$ be
the Ginzburg algebra associated with $(Q',W')$. Let $\ca'$ be the canonical
heart $\ca'$ in $\cd_{fd}\Gamma'$. There are two canonical equivalences
\cite{KellerYang11}
\[
\cd\Gamma' \to \cd\Gamma
\]
given by functors $\Phi_{\pm}$ related by
\[
\tw_{S_k} \circ \Phi_- \iso \Phi_+ \;.
\]
If we put $P_i=e_i \Gamma$, $i\in Q_0$, and similarly for $\Gamma'$,
then both $\Phi_+$ and $\Phi_-$ send $P'_i$
to $P_i$ for $i\neq k$; the images of $P'_k$ under
the two functors fit into triangles
\[
\xymatrix{P_k \ar[r] & \bigoplus_{k\to i} P_i \ar[r] &  \Phi_-(P'_k) \ar[r] &  \Sigma P_k}
\]
and
\[
\xymatrix{\Sigma^{-1} P_k \ar[r] & \Phi_+(P'_k) \ar[r] & \bigoplus_{j\to k} P_j \ar[r] &  P_k}.
\]
The functors $\Phi_{\pm}$ send $\ca'$ onto the hearts $\mu_k^{\pm}(\ca)$ of
two new t-structures. These can be described in terms of
$\ca$ and the subcategory $\add S_k$ as follows
(\cf figure~\ref{fig:mut-hearts}): Let
$S_k^\perp$ be the right orthogonal subcategory of $S_k$ in $\ca$,
whose objects are the $M$ with $\Hom(S_k,M)=0$. Then
$\mu_k^+(\ca)$ is formed by the objects $X$ of $\cd_{fd}\Gamma$
such that the object $H^0(X)$ belongs to $S_k^\perp$, the object $H^1(X)$ belongs
to $\add S_k$ and $H^p(X)$ vanishes for all
$p\neq 0,1$. Similarly, the subcategory $\mu_k^-(\ca)$ is
formed by the objects $X$ such that the object $H^0(X)$ belongs to
the left orthogonal subcategory $^\perp\!S_k$, the object $H^{-1}(X)$
belongs to $\add(S_k)$ and $H^p(X)$ vanishes for all
$p \neq -1,0$. The subcategory $\mu_k^{+}(\ca)$ is
the {\em right mutation of $\ca$} and $\mu_k^-(\ca)$ is
its {\em left mutation}.
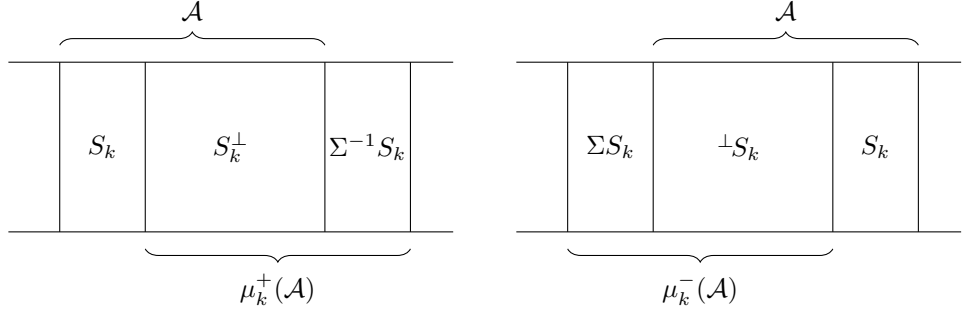
\begin{figure}
\begin{center}
\begin{tikzpicture}[scale=0.45]
\draw (0,4)--(13,4);
\draw (0,9)--(13,9);
\draw (1.5,4)--(1.5,9);
\draw (4,4)--(4,9);
\draw (9.25,4)--(9.25,9);
\draw (11.75,4)--(11.75,9);
\draw (2.75,6.5) node {$S_k$};
\draw (6.5,6.5) node {$S_k^\perp$};
\draw (10.5,6.5) node {$\Sigma^{-1}S_k$};
\draw [decorate, decoration={brace,amplitude=5pt}] (1.5,9.5)--(9.25,9.5)
node [midway, above=6pt] {$\ca$};
\draw [decorate, decoration={brace,amplitude=5pt}] (11.75,3.5)--(4, 3.5)
node [midway, below=6pt] {$\mu_k^+(\ca)$};
\end{tikzpicture}
\quad\quad
\begin{tikzpicture}[scale=0.45]
\draw (0,4)--(13,4);
\draw (0,9)--(13,9);
\draw (1.5,4)--(1.5,9);
\draw (4,4)--(4,9);
\draw (9.25,4)--(9.25,9);
\draw (11.75,4)--(11.75,9);
\draw (2.75,6.5) node {$\Sigma S_k$};
\draw (6.5,6.5) node {$^\perp\!S_k$};
\draw (10.5,6.5) node {$S_k$};
\draw [decorate, decoration={brace,amplitude=5pt}] (4,9.5)--(11.75,9.5)
node [midway, above=6pt] {$\ca$};
\draw [decorate, decoration={brace,amplitude=5pt}] (9.25,3.5)--(1.5, 3.5)
node [midway, below=6pt] {$\mu_k^-(\ca)$};
\end{tikzpicture}
\caption{Right and left mutation of a heart}\label{fig:mut-hearts}
\end{center}
\end{figure}

The construction of these subcategories is very similar to that of the
tilts of \cite{HappelReitenSmaloe96} but notice that
in contrast to the setting of \cite{HappelReitenSmaloe96},
we usually do not have an equivalence between the bounded
derived category $\cd^b(\ca)$ and the ambient Calabi-Yau
category $\cd_{fd}\Gamma$. By construction, we have
\[
\tw_{S_k}(\mu_k^-(\ca)) = \mu_k^+(\ca).
\]
Since the categories $\ca$ and $\mu^\pm(\ca)$ are hearts of
bounded, non degenerate $t$-structures on $\cd_{fd}\Gamma$, their
Grothendieck groups identify canonically with that of
$\cd_{fd}\Gamma$. They are endowed with canonical bases
given by the simples. Those of $\ca$ identify with
the simples $S_i$, $i\in Q_0$, of $\nil(\cp(Q,W))$.
The simples of $\mu_k^+(\ca)$ are $\Sigma^{-1} S_k$,
the simples $S_i$ of $\ca$ such that $\Ext^1(S_k, S_i)$
vanishes and the objects $\tw_{S_k}(S_i)$ where
$\Ext^1(S_k, S_i)$ is of dimension $\geq 1$. By applying
$\tw_{S_k}^{-1}$ to these objects we obtain the
simples of $\mu_k^-(\ca)$.

\subsection{Comparison of the invariants} \label{ss:comparison-of-invariants}
Let us keep the notations of the preceding sections. Let us assume moreover
that the refined DT-invariants $\E_{Q,W}$ and $\E_{Q',W'}$ are well-defined.

The two realizations $\mu_k^\pm(\ca)$ of $\ca'$ as a
subcategory of $\cd_{fd}\Gamma$ yield two ways of
comparing $\E_{Q,W}$ with $\E_{Q',W'}$. Let us consider the category
$\mu_k^+(\ca)$, which is equivalent via $\Phi_+$ to $\ca'$.
We have the algebras
\[
\hat{\A}_Q=\hat{\A}_{K_0^+(\ca)} \quad\mbox{and}\quad
\hat{\A}_{K_0^+(\mu_k^+(\ca))}.
\]
associated with the positive cones of the Grothendieck
groups of $\ca$ and $\mu_k^+(\ca)$.
They have a common subalgebra $\hat{\A}_{[S_k^\perp]}$ associated with
the cone $[S_k^\perp]$ formed by the classes of the objects in
the right orthogonal subcategory of $S_k$.

We choose a stability function $Z$
on $K_0(\ca)$ such that
$S_k$ has the strictly largest phase among all semi-stable objects.
Then we obtain a factorization
\begin{equation} \label{eq:fact-right-orig}
\E_{Q,W} = \E_{S_k} \E_{S_k^\perp} \ko
\end{equation}
where the second factor is associated with
the right orthogonal category $S_k^\perp$ (if $Z$ is
discrete, this factor is the product of the dilogarithms
corresponding to the stable objects in $S_k^\perp$).
Now via the canonical identifications
\[
K_0(\ca) = K_0(\cd_{fd}\Gamma) = K_0(\mu_k^+(\ca)) \ko
\]
we can use the same stability function $Z$ for $\mu^+_k(\ca)$
and then the object $\Sigma^{-1}S_k$ has the strictly smallest phase among all
semi-stable objects, cf. figure~\ref{fig:wall-crossing}.
Moreover, one checks that an object $X$ of $S_k^\perp$
is $Z$-stable (resp. $Z$-semi-stable) in $\ca$ iff it is $Z$-stable
(resp. $Z$-semi-stable) in $\mu^+_k(\ca)$.
Therefore, we obtain a factorization
\begin{equation} \label{eq:fact-right-mut}
\phi_+(\E_{Q', W'}) = \E_{S_k^\perp} \E_{\Sigma^{-1} S_k} \ko
\end{equation}
where $\phi_+: \hat{\A}_{Q'} \to \hat{\A}_{K_0^+(\mu_k^+(\ca))}$ is induced by
the isomorphism $K_0(\ca') \to K_0(\mu^+_k(A))$ provided
by the equivalence $\Phi_+$. Thus, we have
\begin{equation} \label{eq:phi-plus}
\phi_+(y'_i) = \left\{ \begin{array}{ll} y_k^{-1} & \mbox{ if $i=k$} \\
y_i & \mbox{ if there is no arrow $i\to k$ in $Q$} \\
q^{-m^2/2} y_i y_k^m & \mbox{ if there are $m\geq 1$ arrows $i\to k$ in $Q$}
\end{array} \right.
\end{equation}
\begin{figure}
\begin{center}
\begin{tikzpicture}[scale=0.4]
\draw[dashed] (1.5, 2.5) ellipse (3.5 and 2);
\draw (5,4) node {$S_k^\perp$};
\draw[dotted,->] (-4,0) -- (7,0);
\draw[dotted,->] (0,-2) -- (0,7);
\draw[->] (0,0)--(-4,1) node [left] {$Z(S_k)$};
\draw[->] (0,0)--(-1,2);
\draw[->] (0,0)--(0.5,3);
\draw[->] (0,0)--(2,4);
\draw[->] (0,0)--(4,2);
\draw[->] (0,0)--(4,-1) node [right] {$Z(\Sigma^{-1}S_k)$};
\end{tikzpicture}
\quad\quad
\begin{tikzpicture}[scale=0.4]
\draw[dashed] (-1.5, 2.5) ellipse (3.5 and 2);
\draw (-5.25,4) node {$^\perp\!S_k$};
\draw[dotted,->] (-4,0) -- (7,0);
\draw[dotted,->] (0,-2) -- (0,7);
\draw[->] (0,0)--(4,1) node [right] {$Z(S_k)$};
\draw[->] (0,0)--(1,2);
\draw[->] (0,0)--(-0.5,3);
\draw[->] (0,0)--(-2,4);
\draw[->] (0,0)--(-4,2);
\draw[->] (0,0)--(-4,-1) node [left] {$Z(\Sigma S_k)$};
\end{tikzpicture}
\end{center}
\caption{Stable objects in a heart and its right and left mutations}
\label{fig:wall-crossing}
\end{figure}
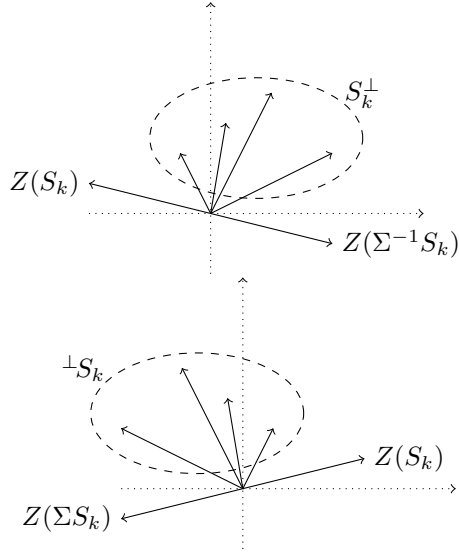
By combining equations~(\ref{eq:fact-right-orig}) and (\ref{eq:fact-right-mut})
we obtain the equality
\begin{equation} \label{eq:wall-crossing}
\E^{-1}_{S_k}\, \E_{Q,W} = \E_{S_k^\perp} = \phi_+(\E_{Q',W'})\, \E^{-1}_{\Sigma^{-1} S_k}
\end{equation}
in $\hat{\A}_{[S_k^\perp]}$. Similarly, one obtains the equality
\begin{equation} \label{eq:left-wall-crossing}
\E_{Q,W}\E_{S_k}^{-1} = \E_{^\perp\!S_k}=\E_{\Sigma S_k}^{-1} \phi_-(\E_{Q',W'})
\end{equation}
in the algebra $\hat{\A}_{[^\perp\!S_k]}$.

\subsection{The rational case} \label{ss:rational-case}
Let us keep the notations of
section~\ref{ss:comparison-of-invariants}. Notice that the series $\E_{Q,W}$ is
invertible in the algebra $\hat{\A}_{Q}$ and so the conjugation
with this series, \ie the automorphism
\[
\Ad(\E_{Q,W}) : u \mapsto \E_{Q,W}\, u \,\E_{Q,W}^{-1}  \ko
\]
is well defined. An element $u$ of $\hat{\A}_Q$ is {\em rational}
if there is a non zero element $s$ of the (non completed) algebra
$\A_Q$ such that $su$ belongs to $\A_Q$. The rational elements
of $\hat{\A}_Q$ form a subalgebra since the non zero elements
of $\A_Q$ satisfy the Ore conditions (by the appendix to
\cite{BerensteinZelevinsky05}). 
The invariant $\E_{Q,W}$ is {\em rational} if
the automorphism $\Ad(\E_{Q,W})$ preserves
the subalgebra of rational elements of $\hat{\A}_Q$.
Clearly, this holds iff $\Ad(\E_{Q,W})(y_i)$
is rational for each $i\in Q_0$. In this case,
the automorphism $\Ad(\E_{Q,W})$ extends to
an automorphism of the (non commutative)
{\em fraction field $\Frac(\A_Q)$} of the (non completed)
algebra $\A_Q$, which is obtained from $\A_Q$
by localizing at the set of all non zero elements.

\begin{lemma} \label{lemma:rational-DT-invariant}
If $\ca$ admits a discrete stability function
with finitely many stables, then $\E_{Q,W}$ is rational.
\end{lemma}

\begin{proof} Under the hypothesis, the automorphism
$\Ad(\E_{Q,W})$ is a composition of finitely many automorphisms
of the form $\Ad(\E(y^\alpha))$ and so it is enough to check
that such an automorphism preserves $\Frac(\A_Q)$. Indeed,
one checks from the definition that
\[
\E(y) (1+q^{1/2} y) = \E(qy).
\]
Therefore, for each $m\geq 0$, we have
\begin{equation} \label{eq:E-formula1}
\E(q^m y)\E(y)^{-1} = \prod_{j=1}^m (1+q^{-1/2+j}y)
\end{equation}
and
\begin{equation} \label{eq:E-formula2}
\E(q^{-m} y) \E(y)^{-1} = \prod_{j=1}^m(1+q^{-m-1/2+j}y)^{-1}.
\end{equation}
Now suppose that we have $\alpha$, $\beta$ such that $m=\lambda(\alpha,\beta)$ and
therefore $y^\alpha y^\beta = q^m y^\beta y^\alpha$ and
\[
y^{-\beta} y^\alpha y^\beta = q^m y^\alpha.
\]
Then we find
\begin{align*}
\E(y^\alpha) y^\beta \E(y^\alpha)^{-1} &= y^\beta y^{-\beta} \E(y^\alpha) y^\beta \E(y^\alpha)^{-1} \\
&= y^\beta\, \E(q^m y^\alpha) \E(y^\alpha)^{-1}
\end{align*}
By the above formulas~(\ref{eq:E-formula1}) and (\ref{eq:E-formula2}),
this expression does belong to $\Frac(\A_Q)$.
\end{proof}

For example, if $Q$ is a Dynkin quiver (and so $W=0$), then
$\E_{Q,W}$ is rational since we can find a discrete stability
function whose stables are the simples. More generally, the
same argument shows that $\E_{Q,W}$ is rational for any
acyclic quiver $Q$. Other examples of rational $\E_{Q,W}$
arise  from (\cf \cite{Keller11} for details):
\begin{itemize}
\item the quivers in Kontsevich-Soibelman's class $\cp$,
\cf section~8.4 of \cite{KontsevichSoibelman08},
\item the quivers with potential $(Q,W)$
associated in Proposition~5.12 of \cite{Keller10a} with pairs of acyclic
quivers,
\item the quivers with potential appearing in the
work of Buan-Iyama-Reiten-Scott \cite{BuanIyamaReitenScott09}
respectively Geiss-Leclerc-Schr\"oer \cite{GeissLeclercSchroeer08a},
\cf \cite{AmiotReitenTodorov11}.
\end{itemize}
On the other hand, the quiver
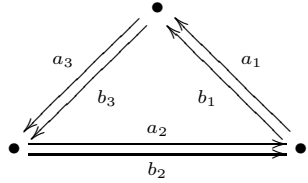
\[
\xymatrix@R=1.5cm@C=1.5cm{ & \bullet \ar@<-2pt>[dl]_{a_3} \ar@<2pt>[dl]^{b_3} & \\
\bullet \ar@<2pt>[rr]^{a_2} \ar@<-2pt>[rr]_{b_2} &  & \bullet \ar@<2pt>[ul]^{b_1} \ar@<-2pt>[ul]_{a_1} }
\]
with the potential $W=a_1 a_2 a_3 + b_1 b_2 b_3$ does not yield
a rational invariant $\E_{Q,W}$, \cf section~8.4 of \cite{KontsevichSoibelman08}.

\subsection{The interwiners}  \label{ss:intertwiners}
Let us keep the hypotheses
of section~\ref{ss:comparison-of-invariants} and suppose moreover
that $\E_{Q,W}$ is rational. Let $\Sigma: \Frac(\A_Q) \to \Frac(\A_Q)$ be
the automorphism mapping each $y^\alpha$ to $y^{-\alpha}$. Define
the {\em non commutative DT-invariant} to be the
automorphism
\[
DT_{Q,W} = \Ad(\E_{Q,W})\circ \Sigma
\]
of $\Frac(\A_Q)$.
This invariant is preserved under left and right mutations up to conjugacy.
Indeed, if we consider right mutation in the setting of
section~\ref{ss:comparison-of-invariants}, the equality
\begin{equation} \label{eq:wall-crossing2}
\E^{-1}_{S_k}\, \E_{Q,W} = \E_{S_k^\perp} = \phi_+(\E_{Q',W'})\, \E^{-1}_{\Sigma^{-1} S_k}
\end{equation}
yields
\[
\Ad(\E(y_k))^{-1} \Ad(\E_{Q,W}) = \phi_+ \Ad(\E_{Q',W'}) \phi_+^{-1} \Ad(\E(y_k^{-1}))^{-1}
\]
and if we pre-compose with $\Sigma$ and post-compose with $\Ad(\E(y_k))$, we obtain
\[
\Ad(\E_{Q,W}) \circ \Sigma =  \Ad(\E(y_k))\phi_+  \Ad(\E_{Q',W'})\,\Sigma \,
 \phi_+^{-1}\Ad(\E(y_k))^{-1}
\]
so that conjugation by the `right intertwiner'
\begin{equation} \label{eq:FG-intertwiner}
\Ad(\E(y_k)) \circ \phi_+ = \phi_+\circ \Ad(\E({y'}_k^{-1})): \Frac(\A_{Q'}) \to \Frac(\A_Q)
\end{equation}
transforms $DT_{Q',W'}$ into $DT_{Q,W}$. This is exactly the quantum mutation operator defined by
Fock and Goncharov in section~3.1 of \cite{FockGoncharov09a}
(they write $q$ for the indeterminate we denote by $q^{1/2}$).
If we consider left mutation, then the equation~\ref{eq:left-wall-crossing} yields
the `left intertwiner'
\begin{equation}
\label{eq:FG-left-intertwiner}
\phi_- \circ \Ad(\E(y_k))^{-1}.
\end{equation}
Remarkably, as we will see in lemma~\ref{lemma:mutation-involutive} below, the
right and left intertwiners (\ref{eq:FG-intertwiner}) and (\ref{eq:FG-left-intertwiner})
are {\em equal}.  Explicitly, if $r\geq 0$ is the number of arrows $k\to i$
and $s\geq 0$ the number of arrows $i \to k$ in $Q$, they are both given by
\begin{equation} \label{eq:intertwiner+}
\Ad(\E(y_k))\circ \phi_+(y'_i) =
\left\{ \begin{array}{ll} y_i \prod_{j=1}^r (1+q^{-1/2+j} y_k) & \mbox{if $r> 0$}\\
y_i & \mbox{if $r=s=0$} \\
y_i y_k^s q^{-s^2/2} \prod_{j=1}^s (1+q^{1/2-j}y_k)^{-1} &  \mbox{if $s> 0$}
\end{array} \right.
\end{equation}
When $q^{1/2}$ is specialized to $1$ and $Q$ replaced with $Q^{op}$, these formulas
yield the transformation rule for $Y$-seeds in the sense of Fomin-Zelevinsky,
\cf \cite{FominZelevinsky07}. Notice that in deducing formula~(\ref{eq:FG-intertwiner}),
we chose to pre-compose with $\Sigma$. If we post-compose with $\Sigma$,
we obtain a variant of the intertwiner which fortunately
has the same good properties and which, upon specialization of
$q^{1/2}$ to $1$, yields Fomin-Zelevinsky's transformation rule
for $Y$-seeds (without replacing $Q$ by $Q^{op}$).

\subsection{Mutation is an involution} \label{ss:mutation-involutive}
In the setting of section~\ref{ss:comparison-of-categories}, one
checks easily that
\[
\mu_k^+(\mu_k^-(\ca))= \ca.
\]
We also know that $\mu_k^+(\ca) = \tw_{S_k}(\mu^-_k(\ca))$.
Thus, we obtain
\[
\mu_k^+(\mu_k^+(\cb)) = \tw_{S_k}(\cb)
\]
for $\cb=\mu^-_k(\ca)$. Thus, mutation of hearts is an
involution only up to the braid group action. Remarkably,
the mutation intertwiner~(\ref{eq:FG-intertwiner}) `squares'
to the identity (as do its variants) by the following
lemma (\cf Lemma~3.7 of \cite{FockGoncharov09}),
where we only write `unprimed' variables to avoid
clutter in the notation.
\begin{lemma} \label{lemma:mutation-involutive}
\begin{itemize} \item[a)] We have
\[
\Ad(\E(y_k))\circ \Ad(\E(y_k^{-1})) = t_k^{-1}
\]
where $t_k : K_0(\ca)\to K_0(\ca)$ is induced by
the twist functor $\tw_{S_k}$.
\item[b)] The composition
\[
\xymatrix{\Frac(\A_{\mu_k^2(Q)}) \ar[rr]^{\phi_+\circ \Ad(\E(y_k))} & &\Frac(\A_{Q'}) \ar[rr]^{\phi_+\circ \Ad(\E(y_k))} & & \Frac(\A_Q) }
\]
is the identity.
\item[c)] The right and left intertwiners $\Ad(\E(y_k)) \circ \phi_+$ and
$\phi_- \circ \Ad(\E(y_k))^{-1}$ are equal.
\end{itemize}
\end{lemma}

\begin{proof} a) If $y_k y_i = q^m y_i y_k$, one computes, as
in Lemma~\ref{lemma:rational-DT-invariant}, that
\begin{align*}
\Ad(\E(y_k)) \Ad(\E(y_k^{-1})(y_i) &=
  y_i \E(q^m y_k) \E(y_k)^{-1} \E(q^{-m} y_k^{-1}) \E(y_k^{-1})^{-1}  \\
  &= y^{e_i +m e_k} = t_k^{-1}(y_i).
\end{align*}
b) and c) are an easy consequences.
\end{proof}

\section{Compositions of mutations}
\label{s:compositions-of-mutations}

\subsection{The groupoid of cluster collections} \label{ss:groupoid-of-cluster-collections}
To state identities
between longer compositions of intertwiners~(\ref{eq:FG-intertwiner}), we
now introduce a suitable groupoid (a category where all morphisms are
invertible). We fix a quiver with potential $(Q,W)$ and
work in the setup of section~\ref{ss:setup}.
A {\em cluster collection} \cite{KontsevichSoibelman08}
is a sequence of objects $S'_1$, \ldots, $S'_n$ of $\cd_{fd}\Gamma$
such that
\begin{itemize}
\item[a)] the $S'_i$ are spherical;
\item[b)] for $i\neq j$, the graded space $\Ext^*(S'_i,S'_j)$
vanishes or is concentrated either in degree $1$ or in degree $2$;
\item[c)] the $S'_i$ generate the triangulated category $\cd_{fd}\Gamma$.
\end{itemize}
One can show \cite{RickardRouquier10} \cite{KellerNicolas10}
that in this case, the closure $\ca'$ of the $S'_i$
under iterated extensions is the heart of a non degenerate bounded $t$-structure
on $\cd_{fd}\Gamma$ and that the simples of $\ca'$ are the
$S'_i$ (up to isomorphism). On the other hand, if $\ca'$ is the heart
of a non degenerate bounded $t$-structure on $\cd_{fd}\Gamma$, then
the simples $(S'_1, \ldots, S'_n)$ will satisfy c) but not necessarily
a) and b). If they do, we call $\ca'$ a {\em cluster heart}. In this
way, we obtain a bijection between cluster hearts and permutation
classes of cluster collections, \cf~\cite{KellerNicolas10}.

The {\em groupoid of cluster collections $\ccl=\ccl_Q$} has as objects the
cluster collections $S'$ reachable from the sequence $S=(S_1, \ldots, S_n)$
of the simples of the initial cluster heart $\ca$ by a sequence
\[
\xymatrix{S=S^{(0)} \ar[r] & S^{(1)} \ar[r] & \ldots \ar[r] & S^{(N)} = S'}
\]
of (positive and negative) mutations and permutations, where all
the intermediate sequences $S^{(i)}$ are cluster collections. The
morphisms of $\ccl$ are the formal compositions of mutations and
permutations subject
to the relations valid in the symmetric group and the relations
\[
\sigma \circ \mu^\eps_k = \mu^\eps_{\sigma(k)} \circ \sigma
\]
for all permutations $\sigma$ and mutations $\mu_k^\eps$.
For example, for the quiver $Q:1\to 2$, with the notations of
section~\ref{ss:Reineke-s-identities}, it is easy to check
that we have the following two morphisms from
$(S_1, S_2)$ to $(\Sigma^{-1}S_2, \Sigma^{-1}S_1)$ in $\ccl$:
\begin{align}
\label{eq:ccg-m1} & \xymatrix{(S_1,S_2) \ar[r]^-{\mu_1^+} & (\Sigma^{-1}S_1, S_2)\ar[r]^-{\mu_2^+} &
(\Sigma^{-1}S_1, \Sigma^{-1}S_2) \ar[r]^-\tau & (\Sigma^{-1}S_2, \Sigma^{-1}S_1)} \\
\label{eq:ccg-m2} & \xymatrix{(S_1,S_2) \ar[r]^-{\mu_2^+} & (P_2, \Sigma^{-1}S_2) \ar[r]^-{\mu_1^+} &
(\Sigma^{-1}P_2, S_1) \ar[r]^-{\mu_2^+} & (\Sigma^{-1}S_2, \Sigma^{-1} S_1)}
\end{align}

Given a cluster collection $S'=(S'_1, \ldots, S'_n)$ its quiver
$Q_{S'}$ has the vertex set $\{1, \ldots, n\}$ and the
number of arrows from $i$ to $j$ equals the dimension
of $\Ext^1(S'_j,S'_i)$. Using the intertwiners~(\ref{eq:FG-intertwiner})
and the natural action of the permutation groups, we clearly
obtain a functor
\[
\FG: \ccl^{op} \to \sf \ko
\]
to the groupoid $\sf$ of skew fields which takes a cluster collection
$S'$ to $\Frac(\A_{Q_{S'}})$.
The following theorem is a corollary of the theory of cluster algebras and their (additive)
categorification as developed by Fomin-Zelevinsky,
Berenstein-Fomin-Zelevinsky, Fock-Goncharov, Derksen-Weyman-Zelevinsky \ldots .

\begin{theorem} \label{thm:main-thm} The image of a morphism $\alpha: S \to S'$ of
the groupoid of cluster collections under the functor $\FG$ only
depends on the orbit of $S'$ under the braid group $\Braid(Q)$.
\end{theorem}

We will sketch a proof in section~\ref{ss:proof-main-thm} below.
Notice that already the statement that the image of $\alpha$ only
depends on $S'$ is very strong. By the easy Lemma~\ref{lemma:mutation-involutive},
it implies the statement of the theorem. As an example, consider
the two morphisms~(\ref{eq:ccg-m1}) and (\ref{eq:ccg-m2}). By the theorem,
they yield the equality
\begin{equation} \label{eq:example-ccg}
\Ad(\E(y_1)) \phi_1 \Ad(\E(y_2)) \phi_2 \, \tau = \Ad(\E(y_2))\phi_2 \Ad(\E(y_1)) \phi_1 \Ad(\E(y_2)) \phi_2
\end{equation}
in the groupoid $\sf$. Notice that the symbols $\phi_1$ and $\phi_2$ denote different
maps depending on the source and target fields. They are given, in the order
of occurrence above, by the matrices
\[
\left[\begin{array}{cc} -1 & 0 \\ 0 & 1 \end{array}\right] \ko
\left[\begin{array}{cc} 1 & 0 \\ 0 & -1 \end{array}\right] \ko
\left[\begin{array}{cc} 1 & 0 \\ 1 & -1 \end{array}\right] \ko
\left[\begin{array}{cc} -1 & 1 \\ 0 & 1 \end{array}\right] \ko
\left[\begin{array}{cc} 1 & 0 \\ 1 & -1 \end{array}\right].
\]
Thus, we have
\begin{align}
\label{eq:a2-dt1} \Ad(\E(y_1))\phi_1 \Ad(\E(y_2)) \phi_2 \tau & = \Ad(\E(y_1)\E(y_2)) \phi_1\phi_2 \tau \mbox{ and } \\
\label{eq:a2-dt2} \Ad(\E(y_2))\phi_2 \Ad(\E(y_1)) \phi_1 \Ad(\E(y_2)) \phi_2 &= \Ad(\E(y_1)\E(q^{-1/2} y_1 y_2) \E(y_1)) \phi_2 \phi_1 \phi_2.
\end{align}
Since we have $\phi_1 \phi_2 \tau = \phi_2\phi_1\phi_2$,
the equality~(\ref{eq:example-ccg}) is in fact a consequence
of the pentagon identity~(\ref{eq:SFK-identity})
\[
\E(y_1) \E(y_2) = \E(y_2) \E(q^{-1/2} y_1 y_2) \E(y_1).
\]

\subsection{Action of autoequivalences of the derived category} \label{ss:action-autoequivalences-der-cat}
Let $F: \cd_{fd}\Gamma \to \cd_{fd}\Gamma$ be a triangle equivalence such
that $F$ is {\em reachable}, \ie there is a sequence of (per-)mutations
linking the initial cluster collection $S=(S_1, \ldots, S_n)$ to $FS$.
Thus, the collection $FS$ still belongs to $\ccl$ and of course,
so does $FS'$ for any other cluster collection $S'$ in $\ccl$.
With $F$, we will associate a canonical automorphism $\zeta(F)$
of the functor
\[
\FG: \ccl^{op} \to \sf.
\]
Indeed, for any cluster collection $S'$, we have an
isomorphism of quivers $Q_{FS'} = Q_{S'}$ given by the
bijection $FS'_i \mapsto S'_i$ on the set of vertices. Thus,
we have a canonical isomorphism of skew fields
\[
\FG(FS') \iso \FG(S').
\]
We define $\zeta(F)(S') : \FG(S') \to \FG(S')$ to be the
composition of this isomorphism with $\FG(\alpha)$ for any
morphism $\alpha: S' \to FS'$. It is immediate from
Theorem~\ref{thm:main-thm}, that $\zeta(F)$ is indeed
an automorphism of $\FG(S')$, and that $\zeta$ defines
a homomorphism from the group of (isomorphism classes of)
reachable autoequivalences of $\cd_{fd}\Gamma$ to the
group of automorphisms of the functor $\FG$.
The following theorem is based on Nagao's
ideas \cite{Nagao10}.

\begin{theorem} Suppose that the inverse suspension functor
$\Sigma^{-1}: \cd_{fd}\Gamma \to \cd_{fd}\Gamma$ is reachable.
Then $(Q,W)$ is Jacobi-finite. If moreover conjecture~\ref{conj:KS} 
holds for $(Q,W)$ so that the refined 
DT-invariant $\E_{Q,W}$ is well-defined, then it is rational and 
$DT_{Q,W}$ equals $\zeta(\Sigma^{-1})$.
\end{theorem}

For example, consider the quiver $Q: 1\to 2$. Then the
morphism $\alpha=\mu_2^+ \mu_1^+$ of (\ref{eq:ccg-m1}) shows
that $\Sigma^{-1}$ is reachable. Now we use the equality~(\ref{eq:a2-dt1})
and the fact that the composition $\phi_1\phi_2$
of $\phi_1$ with $\phi_2$ in (\ref{eq:a2-dt1}) equals $\Sigma$ to deduce that,
in accordance with the theorem, we have
\[
DT_{Q,W} = \Ad(\E(y_1)\E(y_2)) \circ \Sigma = \FG(\alpha) = \zeta(\Sigma^{-1}).
\]

\subsection{The groupoid of cluster tilting sequences} In view of
Theorem~\ref{thm:main-thm}, it is natural to try and `factor out' the braid
group action on the category $\cd_{fd}\Gamma$. This is, in a
certain sense, what is achieved by the passage to the cluster
category. For simplicity, let us assume that $(Q,W)$ is Jacobi-finite
(the general case can be treated using Plamondon's
results \cite{Plamondon10a} \cite{Plamondon10b}).
The cluster category $\cc_{Q,W}$ is defined \cite{Amiot09} as the triangle quotient
$\per(\Gamma)/\cd_{fd}(\Gamma)$. It is a triangulated category with
finite-dimensional morphism spaces which is $2$-Calabi-Yau and
admits the image $T$ of $\Gamma$ as a {\em cluster-tilting object},
\ie we have $\Ext^1(T,T)=0$ and any object $X$ such that
$\Ext^1(T,X)=0$ belongs to the category $\add(T)$ of direct summands
of finite direct sums of copies of $T$. For a cluster tilting
object $T'$ the quiver $Q_{T'}$ is the quiver of the endomorphism
algebra of $T'$. A {\em cluster tilting sequence}
is a sequence $(T'_1, \ldots, T'_n)$ of pairwise non isomorphic
indecomposables of $\cc_{Q,W}$ whose direct sum is a cluster
tilting object $T'$ whose associated quiver $Q_{T'}$ does not
have loops or $2$-cycles. There is a canonical mutation operation on
all cluster tilting objects, \cf \cite{IyamaYoshino08}. It yields
a partially defined mutation operation on the cluster-tilting sequences.
The {\em groupoid of cluster-tilting sequences $\clt=\clt_Q$}
has as objects the cluster-tilting sequences of $\cc_{Q,W}$ which
are reachable from the image $T=(T_1, \ldots, T_n)$ of
the sequence of the dg modules $e_i \Gamma$, $i\in Q_0$.
Its morphisms are defined as formal compositions
of mutations and permutations as for the groupoid
of cluster collections. The following theorem proved in \cite{KellerNicolas11}
yields a link between the groupoids $\clt$ and $\ccl$.

\begin{theorem}[Keller-Nicol\'as \cite{KellerNicolas11}]
\label{thm:Keller-Nicolas}
There is a canonical bijection from the set of
$\Braid(Q)$-orbits of cluster collections in $\cd_{fd}\Gamma$
to the set of cluster tilting sequences in $\cc_{Q,W}$. It
is compatible with mutations and permutations and preserves
the quivers.
\end{theorem}

The bijection is based on the exact sequence of triangulated
categories
\[
\xymatrix{0 \ar[r] & \cd_{fd}\Gamma \ar[r] & \per(\Gamma) \ar[r] & \cc_{Q,W} \ar[r] & 0}
\]
Namely, we define a {\em silting sequence} to be a sequence
$(P'_1, \ldots, P'_n)$ of objects in $\per(\Gamma)$ such that
\begin{itemize}
\item[a)] $\Hom(P'_i, \Sigma^p P'_j)$ vanishes for all $p>0$,
\item[b)] the $P'_i$ generate $\per(\Gamma)$ as a triangulated category and
\item[c)] the quiver of the subcategory whose objects are the $P_i'$
does not have loops nor $2$-cycles.
\end{itemize}
For such a sequence, the subcategory $\ca'$ of $\cd_{fd}\Gamma$ formed
by the objects $X$ such that $\Hom(P'_i, \Sigma^p X)=0$ for all $i$ and
all $p\neq 0$ is the heart $\ca'$ of a non degenerate bounded t-structure
whose simples form a cluster collection. The map from silting
sequences to cluster collections thus defined is a bijection.
We obtain the bijection of the theorem by composing its
inverse with the map taking a silting sequence to its
image in the cluster category, which one shows to be
a cluster tilting sequence using Theorem~2.1 of
\cite{Amiot09}.

Thanks to the theorem, we can define a functor
\[
\FG: \clt^{op} \to \sf
\]
by sending a cluster-tilting sequence $T'$ to the image under
$\FG: \ccl^{op} \to \sf$ of the corresponding cluster collection $S'$.

\subsection{Action of autoequivalences of the cluster category} \label{ss:action-autoequivalences-cluster-cat}
In analogy with section~\ref{ss:action-autoequivalences-der-cat},
if $(Q,W)$ is Jacobi-finite, one obtains a homomorphism, still denoted
by $\zeta$, from the group
of reachable autoequivalences of the cluster category $\cc_{Q,W}$
to the automorphism group of the induced functor
\[
\FG : \clt_Q \to \sf.
\]
Again, if the suspension functor $\Sigma^{-1}$ of the cluster category
is reachable, we find that $\zeta(\Sigma^{-1})$ equals $DT_{Q,W}$.
In particular, if $\Sigma$
is of finite order $N$ as an autoequivalence of the cluster category,
then the automorphism $DT_{Q,W}$ of $\Frac(\A_Q)$ is of
finite order dividing $N$.
Applications of these ideas include the (quantum version of
the) periodicity theorem for the $Y$-systems (and $T$-systems)
associated with pairs of simply laced
Dynkin diagrams \cite{Keller08c} \cite{Keller10b} \cite{Keller10a}. Indeed,
let $\vec{\Delta}$ and $\vec{\Delta}'$ be alternating orientations
of simply laced Dynkin diagrams, $Q$ the triangle product
$\vec{\Delta}\boxten \vec{\Delta}'$ and $W$ the canonical
potential on $Q$ defined in Proposition~5.12 of \cite{Keller10a}.
Then the cluster category $\cc_{Q,W}$ is equivalent to the
cluster category $\cc_{A\ten_k A'}$ associated \cite{Amiot09} with the
tensor product of the path algebras $A=k\vec{\Delta}$ and
$A'=\vec{\Delta}'$, by Proposition~5.12 of [loc. cit.]. 
Let $\mu_\boxten$ be the sequence of mutations of $Q$ defined
in (3.6.1) of [loc. cit.]. Then by section~7.4 of [loc. cit.],  
the Zamolodchikov autoequivalence
\begin{equation}
\Za = \tau^{-1} \ten \id : \cc_{A\ten_k A'} \to \cc_{A\ten_k A'}
\end{equation}
is reachable and its image under $\zeta$ is $\mu_\boxten$. 
Moreover, if $h$ and $h'$ are the Coxeter numbers of $\Delta$
and $\Delta'$, then we have
\begin{equation} \label{eq:Za1}
\Za^h = \tau^{-h} \ten \id = \Sigma^2 \ten \id = \Sigma^2.
\end{equation}
 Since we also have (\cf the proof of Theorem~8.4 in [loc. cit.])
\begin{equation}
\Za = \id \ten \tau \ko
\end{equation}
we find that
\begin{equation} \label{eq:Za2}
\Za^{h'} = \id \ten \tau^{h'} = \id \ten \Sigma^{-2} = \Sigma^{-2}.
\end{equation}
Notice that this implies in particular that $\Sigma^{-2}$ is reachable
and yields a sequence of mutations whose composition is the
square of the non commutative DT-invariant $DT_{Q,W}$. 
Equation~(\ref{eq:Za2}) confirms equation~(8.19)
of \cite{CecottiNeitzkeVafa10}.
In fact, it is not hard to show that $\Sigma^{-1}$ is also reachable.
By combining equations (\ref{eq:Za1}) and (\ref{eq:Za2}) we
obtain $\Za^{h+h'}=\id$ and thus $\mu_{\boxten}^{h+h'}=\id$.
By applying the functor $\FG$ to this last equation, we
obtain a statement equivalent to the quantum version
of the periodicity for the $Y$-system associated with $(\Delta,\Delta')$.
Notice that in the course of this reasoning, we have also found that
\[
\Sigma^{-2} = \Za^{h'} = \Za^{-h}.
\]
This implies that the non commutative DT-invariant $DT_{Q,W}$
is of order dividing 
\[
2 \frac{h+h'}{\gcd(h,h')} \ko
\]
a fact already present in section~8.3.2 of \cite{CecottiNeitzkeVafa10}.

\subsection{Nearby cluster collections} \label{ss:nearby-cluster-collections}
Let $S'$ be a cluster collection and $\ca'$ the associated heart. We define
$\ca'$ to be a {\em nearby heart} and $S'$ to be a {\em nearby cluster
collection} if there is a torsion pair $(\cu,\cv)$ in $\ca$
(\cf section~\ref{ss:Reineke-s-identities}) such that
$\ca'$ is the full subcategory formed by the objects $X$ of
$\cd_{fd}\Gamma$ such that the object $H^0(X)$ lies in $\cv$, the object
$H^1(X)$ in $\cu$ and $H^p(X)$ vanishes for all $p\neq 0,1$.

\begin{theorem}[Nagao \cite{Nagao10}] \label{thm:plus-minus}
\begin{itemize}
\item[a)] Let $\ca$ be the initial heart and $\ca'$
a nearby heart. Then each simple $S_k'$ of $\ca'$ either
lies in $\ca$ or in $\Sigma^{-1}\ca$.
\item[b)] Each nearby cluster collection $S'$ is determined
by the classes $[S_i']$ of its objects in $K_0(\cd_{fd}\Gamma)$.
\end{itemize}
\end{theorem}

Part a) of the theorem is equivalent to the `sign-coherence' of
the classes $[P_i'] \in K_0(\per\Gamma)$, where the $P_i'$ form
the silting sequence associated with $S'$. The `sign-coherence'
is also proved in \cite{Plamondon10a} and, in another language,
in \cite{DerksenWeymanZelevinsky10}. Part b) is an easy
consequence.

The following theorem goes back to the insight of
Nagao \cite{Nagao10}. In this form,
it follows \cite{Keller11} from the proof of Theorem~2.18 in
\cite{Plamondon10a} and the generalization of
theorem~\ref{thm:Keller-Nicolas} to quivers with potential
which are not necessarily Jacobi-finite, \cf \cite{KellerNicolas11}.

\begin{theorem} \label{thm:nearby-projection}
\begin{itemize}
\item[a)] Each cluster collection $S'$ belongs to the $\Braid(Q)$-orbit of a unique
nearby cluster collection $\rho(S')$.
\item[b)]  If $S'$ and $S''$ are cluster collections related by a mutation,
then $\rho(S')$ and $\rho(S'')$ are related by a mutation. More precisely,
if $S''=\mu^{\pm}_k(S')$ for some sign $\pm$ and some $1\leq k\leq n$,
then $\rho(S'') = \mu_k^\eps(\rho(S'))$, where $\eps=+1$ if the object
$S'_k$ of $S'$ lies in $\ca$ and $\eps=-1$ if $S'_k$ lies in $\Sigma^{-1}\ca$.
\end{itemize}
\end{theorem}

Clearly, a cluster collection is reachable iff each cluster collection
in its $\Braid(Q)$-orbit is reachable. Thus, the reachable nearby cluster
collections form a system of representatives for the $\Braid(Q)$-orbits
in $\ccl$. Let $\ncc$ denote the full subgroupoid of $\ccl$ formed
by the reachable nearby cluster collections. Then the projection restricts to
a functor
\[
\xymatrix{\ncc \ar@{->>}[r] & \ccl/\Braid(Q) }
\]
which is full and yields a bijection between the sets of objects.
The map $S'\mapsto \rho(S')$ yields a (non functorial!) section.
We have the following diagram of groupoids and functors
\[
\xymatrix{
\ncc \ar@{->>}[drr]\ar@{->}[rr] &  & \ccl \ar[d] \ar[dr] & \\
                 &  & \ccl/\Braid(Q) \ar[r] & \sf^{op}.
}
\]
The resulting functor $\ncc^{op} \to \sf$ can be refined
so as to yield {\em identities between products
of series in $\hat{\A}_Q$} (rather than just in the automorphism
group of $\Frac(\A_Q)$): Let $\cu\subset\ca$
be a torsion subcategory associated with a reachable
nearby cluster heart $\ca'$ (so that an object $X$ of
$\cd_{fd}\Gamma$ belongs to $\ca'$ iff the object $H^1(X)$ belongs
to $\cu$, the object $H^0(X)$ belongs to $\cu^\perp$ and
$H^p(X)$ vanishes for all $p\neq 0$). Let $\alpha: S \to S'$
be a morphism of $\ncc$, where $S$ is the initial cluster
collection and $S'$ a cluster collection whose associated
heart is $\ca'$. After permuting the elements of $S'$ we may
assume that no permutations occur in $\alpha$ so that
$\alpha$ is a composition of mutations of nearby cluster
collections
\[
\xymatrix{
S=S^{(0)} \ar[r]^-{\mu_{k_1}} & S^{(1)} \ar[r]^{\mu_{k_2}} & \cdots \ar[r]^{\mu_{k_N}} & S^{(N)}.
}
\]
For $1\leq i\leq N$, let $\beta_i$ be the class $k_i$-th object of
$S^{(i)}$ in $K_0(\cd_{fd}\Gamma)$. By theorem~\ref{thm:plus-minus},
either $\beta_i$ belongs to the positive cone determined by $\ca$
or to its opposite. We put $\eps_i=+1$ in the first case and
$\eps_i=-1$ in the second. Then $\eps_i \beta_i$ is a positive
integer linear combination of the classes $[S_j]$. We write
$\E(\eps_i\beta_i)$ for $\E(y^v)$, where $v\in \N^n$ is
the vector of the coefficients of the decomposition of
$\eps_i \beta_i$ in the basis given by the $[S_j]$. Define
the invertible element $E_{\cu,\alpha}$ of $\A_Q$ by
\[
\E_{\cu,\alpha} = \E(\eps_1 \beta_1)^{\eps_1} \E(\eps_2 \beta_2)^{\eps_2} \cdots \E(\eps_N \beta_N)^{\eps_N}.
\]
\begin{theorem}[\cite{Keller11}] \label{thm:E-U}
The element $\E_{\cu,\alpha}$ does not depend
on the choice of $\alpha$.
\end{theorem}

The proof of the theorem is independent of conjecture~\ref{conj:KS}.
It is based on the work of Nagao \cite{Nagao10},
Plamondon \cite{Plamondon10a} \cite{Plamondon10b},
Derksen-Weyman-Zelevinsky \cite{DerksenWeymanZelevinsky08}
\cite{DerksenWeymanZelevinsky10}, Berenstein-Zelevinsky
\cite{BerensteinZelevinsky05}, \ldots.
We put $\E_\cu=\E_{\cu,\alpha}$ for any $\alpha$. Notice that
if the cluster heart $\Sigma^{-1}\ca$ is reachable, the
corresponding torsion subcategory is $\cu=\ca$.

\begin{theorem}[\cite{Keller11}] \label{thm:reconstruct-refined}
Suppose that the ground field equals $\C$ and
that conjecture~\ref{conj:KS} holds for $(Q,W)$ so that the refined DT-invariant
$\E_{Q,W}$ of (\ref{eq:refined-DT-invariant}) is well-defined.
If the cluster heart $\Sigma^{-1}\ca$ is reachable,
then $\E_{\ca}$ equals $\E_{Q,W}$.
\end{theorem}

\subsection{On the proof of the main theorem} \label{ss:proof-main-thm}
To prove theorem~\ref{thm:main-thm}, by
lemma~\ref{lemma:mutation-involutive}, it suffices to show
that if $\alpha: S \to S'$ is a morphism of $\ccl$ to a reachable
nearby cluster collection, then $\FG(\alpha)$ only depends on $S'$.
The proof \cite{Keller11} of this fact uses
\begin{itemize}
\item[1)] the technique of the `double torus' (\cf \eg \cite{FockGoncharov09a})
to reduce the statement to a statement about seeds in quantum cluster algebras;
\item[2)] Berenstein-Zelevinsky's theorem which states that
the exchange graph of the quantum cluster algebra of a quiver is
canonically isomorphic to the exchange graph of its cluster algebra
(Theorem~6.18 of \cite{BerensteinZelevinsky05});
\item[3)] the expression of the (classical) cluster variables in
terms of quiver Grassmannians first obtained in this generality
by Derksen-Weyman-Zelevinsky \cite{DerksenWeymanZelevinsky10},
\cf also \cite{Nagao10} \cite{Plamondon10a}.
\end{itemize}

Among these three ingredients, the third one is perhaps the deepest.
Let us make it explicit: Let $S'$ be a reachable nearby cluster
collection. After permuting its objects, we may assume that
there is a sequence of mutations transforming the initial
cluster collection $S$ to the given cluster collection $S'$.
This sequence determines a vertex $t$ in the $n$-regular
tree and thus a cluster $(X_i(t), 1\leq i\leq n)$ in the cluster
algebra associated with $Q$, \cf \cite{FominZelevinsky07}.
Following \cite{Nagao10} and \cite{Plamondon10b}, we can
express the cluster variables $X_j(t)$ in terms of the cluster
collection $S'$ as follows: Let $(T'_1, \ldots, T'_n)$ be the
silting sequence associated with $S'$ and $(T_1, \ldots, T_n)$
the initial silting sequence. Define the integers
$g_{ij}$, $1\leq i,j \leq n$, by the equality
\[
[T'_j] = \sum_{i=1}^n g_{ij} [T_i]
\]
in $K_0(\per\Gamma)$. Then we have
\[
X_j(t) = \prod_{i=1}^n x_i^{g_{ij}} \sum_e \chi(\Gr_e(H^1(T'_j)))
\prod_{i=1}^n x_i^{\langle S_i,e\rangle} \ko
\]
where $\Gr_e$ is the Grassmannian of submodules of dimension
vector $e$ and $\chi$ the Euler characteristic (for singular cohomology
with rational coefficients of the underlying topological space).

\subsection{The tropical groupoid} \label{ss:tropical-groupoid}
We will exhibit a groupoid
defined in purely combinatorial terms which, if the potential
$W$ is generic, is isomorphic to
the groupoid of nearby cluster collections (and thus admits
a full surjective functor to the
groupoid of reachable cluster collections modulo the
braid group action and to the groupoid of tilting sequences in the
cluster category). Let $\tilde{Q}$ be the quiver obtained from $Q$ by adding a
new vertex $i'$ and a new arrow $i\to i'$ for each vertex $i$ of $Q$.
The new vertices $i'$ are called {\em frozen} because we never
mutate at them.
The {\em tropical groupoid $\trp=\trp_Q$}
has as objects all the quivers obtained from $\tilde{Q}$ by mutating
at the non frozen vertices $1, \ldots, n$. Its morphisms are formal
compositions of mutations at the non frozen vertices
and permutations of these vertices as in the definition
of the groupoid of cluster tilting sequences in
section~\ref{ss:action-autoequivalences-cluster-cat}.

We construct a morphism of groupoids $\ncc\to\trp$ as follows:
For a reachable nearby cluster collection $S'$, define
the quiver $q(S')$ to have the same vertices as $\tilde{Q}$,
such that the full subquiver on $1, \ldots, n$ is the
$\Ext$-quiver of $S'$, there are no arrows between
frozen vertices, and for each old vertex $i$, the number of
arrows from $i$ to a frozen vertex $j'$ is the integer $b_{ij}$
defined by the equality
\[
[S'_i] = \sum_{j=1}^n b_{ij} [S_j]
\]
in $K_0(\cd_{fd}\Gamma)$. Here, if $b_{ij}$ is negative, we draw
$-b_{ij}$ arrows from $j'$ to $i$. From theorem~\ref{thm:plus-minus},
we deduce the following corollary.

\begin{corollary} \label{cor:tropical-groupoid}
\begin{itemize}
\item[a)]  The quiver $q(S')$ uniquely determines $S'$.
\item[b)]  The map $S' \mapsto q(S')$ underlies a unique isomorphism
of groupoids $\ncc \iso \trp$.
\end{itemize}
\end{corollary}

Now we can give a purely combinatorial version of theorem~\ref{thm:E-U}:
Let $\mathbf{k}=(k_1, \ldots, k_N)$ be a
sequence of vertices of $Q$ (no frozen vertices are allowed to occur).
Let $\mu_{\mathbf{k}}(\tilde{Q})$ be the quiver
\[
\mu_{k_{N}} \mu_{k_{N-1}} \ldots \mu_{k_1}(\tilde{Q}) \ko
\]
and, more generally, for each $1\leq s\leq N$, let $\tilde{Q}(\mathbf{k},s)$
be the quiver
\[
\mu_{k_{s-1}} \mu_{k_{s-2}} \ldots \mu_{k_1}(\tilde{Q}).
\]
Let $B^{\tilde{Q}(\mathbf{k},s)}$ be the antisymmetric matrix associated
with this quiver and let $\beta_s$ be the vector
\[
\beta_s = \sum_{j=1}^n b_{k_s, j'}^{\tilde{Q}(\mathbf{k},s)} e_j
\]
in $\Z^n$. We know from theorem~\ref{thm:plus-minus} that either
all components of $\beta_s$ are non negative or all are non positive.
We put $\eps_s=+1$ in the first case and $\eps_s=-1$ in the second.
Now we define
\[
\E(\mathbf{k}) = \E(\eps_1 \beta_1)^{\eps_1} \E(\eps_2 \beta_2)^{\eps_2} \cdots \E(\eps_N \beta_N)^{\eps_N}.
\]
Let $\mathbf{k}'$ be another sequence of vertices of $Q$.

\begin{theorem} If there is an isomorphism of quivers
\[
\mu_{\mathbf{k}}(\tilde{Q}) \iso \mu_{\mathbf{k}'}(\tilde{Q})
\]
which is the identity on the frozen vertices $j'$, $1\leq j\leq n$,
then we have
\[
\E(\mathbf{k}) = \E(\mathbf{k}')
\]
in $\hat{\A}_Q$.
\end{theorem}
For example, if we apply the theorem to $Q:1 \to 2$ and the sequences
$\mathbf{k}=(1,2)$ and $\mathbf{k}'=(2,1,2)$,
\cf figure~\ref{fig:max-green-sequences}, then all the $\beta_s$
are positive and we find the pentagon
identity~(\ref{eq:SFK-identity}). If we use $\mathbf{k}=(1,2,1)$ and
$\mathbf{k}'=(2,1)$ instead, then for $\mathbf{k}$, the vector $\beta_3$
is negative and we find the identity
\[
\E(y_1) \E(y_2) \E(y_1)^{-1} = \E(y_2) \E(q^{-1/2}y_1 y_2) \ko
\]
which is of course equivalent to (\ref{eq:SFK-identity}).

Let $\tilde{Q}'$ be a quiver of the tropical groupoid $\trp_Q$.
Define a non frozen vertex $i$ of $\tilde{Q}'$ to be {\em green}
if there are no arrows from frozen vertices to $i$ and {\em red}
otherwise. Define a sequence $\mathbf{k}$ of green vertices
of $\tilde{Q}'$ to be {\em maximal} if all non frozen vertices
of $\mu_{\mathbf{k}}(\tilde{Q}')$ are red. In many cases, the
following proposition allows one to construct the refined
DT-invariant combinatorially (\cf also section~5.2, page~49
of \cite{GaiottoMooreNeitzke10}).

\begin{proposition} Suppose the $\tilde{Q}$ admits a maximal
green sequence $\mathbf{k}$. Then the cluster heart $\Sigma^{-1}\ca$
is reachable and, if the refined DT-invariant $\E_{Q,W}$ is
well-defined, it equals $\E(\mathbf{k})$.
\end{proposition}
In figure~\ref{fig:max-green-sequences}, the two maximal
green sequences for the quiver $\vec{A}_2$ are given
(green vertices are encircled).
Examples of classes of quivers $Q$ to which the proposition applies
include those enumerated in section~\ref{ss:rational-case}.
The quiver mutation applet \cite{KellerQuiverMutationApplet}
makes it easy to search for maximal green sequences.
They exist for acyclic quivers, for square products of
acyclic quivers and also for the quivers associated
in \cite{BuanIyamaReitenScott09} with each pair consisting of an
ayclic quiver and a reduced expression of an element
in the Coxeter group associated with its underlying
graph. For this case, a maximal green sequence is
constructed in section~12 of \cite{GeissLeclercSchroeer10}.

\begin{figure}
\[
\def\g#1{\save [].[dr]!C *++\frm{}="g#1" \restore}
\xymatrix{
 & &  *+[o][F-]{1}\g1 \ar[r] \ar[d] & *+[o][F-]{2} \ar[d] & & *+[o][F-]{1}\g2  \ar[d] \ar[dr] & 2 \ar[l] & &  \\
 & &   1'                & 2' & &    1'               & 2' \ar[u]\\
\g3 1 & *+[o][F-]{2} \ar[l] \ar[d]  &  &  &   &    &   & \g4 1 \ar[r] & *+[o][F-]{2} \ar[dl] \\
1'\ar[u] & 2'            &  &  &   &    &   &  1' \ar[u] & 2' \ar[ul] \\
 & & \g5 1 \ar[r] & 2         &  &  \g6 1 & 2 \ar[l] \\
 & &    1' \ar[u] & 2' \ar[u] &  &   1' \ar[ur] & 2' \ar[ul]
 \ar@{->}^{\mu_2} "g1"; "g2"
 \ar@<-20pt>@{->}_{\mu_1} "g1"; "g3"
 \ar@<20pt>@{->}^{\mu_1} "g2"; "g4"
 \ar@<-20pt>@{->}_{\mu_2} "g3"; "g5"
 \ar@<20pt>@{->}^{\mu_2} "g4"; "g6"
 \ar@{-}^{\mbox{\small isom}} "g5"; "g6"
 \ar_{\mu_{12}} "g1"; "g5"
 \ar^{\mu_{212}} "g1"; "g6"
 }
\]
\caption{The two maximal green sequences for $A_2$}
\label{fig:max-green-sequences}
\end{figure}
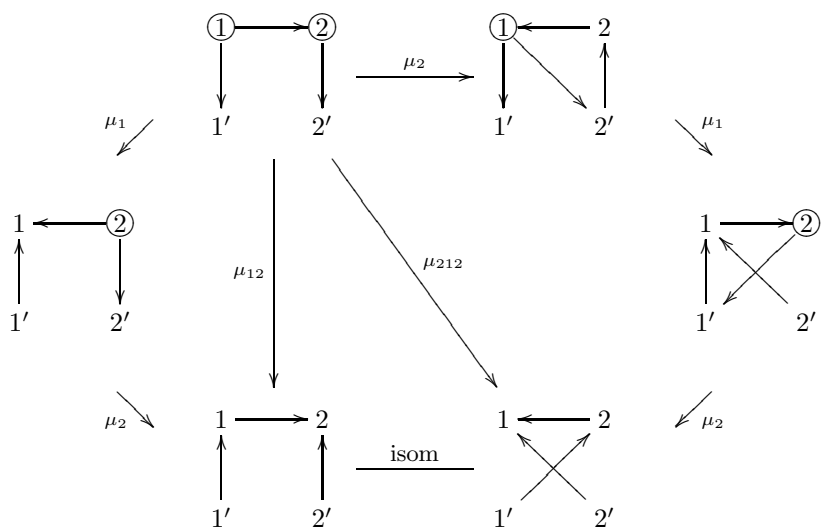



\def\cprime{$'$} \def\cprime{$'$}
\providecommand{\bysame}{\leavevmode\hbox to3em{\hrulefill}\thinspace}
\providecommand{\MR}{\relax\ifhmode\unskip\space\fi MR }
\providecommand{\MRhref}[2]{%
  \href{http://www.ams.org/mathscinet-getitem?mr=#1}{#2}
}
\providecommand{\href}[2]{#2}

\end{document}